\documentclass[11pt,a4paper]{article}
\usepackage{amssymb}
\usepackage{amsmath}
\usepackage{amsthm}
\usepackage{graphicx}
\usepackage{texdraw}
\usepackage{epsfig}
 \allowdisplaybreaks \numberwithin{equation}{section}
\newtheorem{theorem}{Theorem}[section]
\newtheorem{proposition}[theorem]{Proposition}
\newtheorem{lemma}[theorem]{Lemma}
\newtheorem{corollary}[theorem]{Corollary}
\newtheorem{example}[theorem]{Example}
\newtheorem{definition}[theorem]{Definition}
\newtheorem{remark}[theorem]{Remark}

\def\neweq#1{\begin{equation}\label{#1}}
\def\endeq{\end{equation}}

 \font \ninemaius=cmcsc10 scaled 900
\font\nineit=cmti9  \font\svfilt=msbm7
\font\ninebf=cmbx9

\font\nineit=cmti9
\font\ninerm=cmr9

\def \div {\mathop {\rm div}\nolimits}
\def \e {\varepsilon}
\def \u {\overline{u}}

\def \ov {\overline}

\def\qed{{}\hfill{$\square$}\par\medskip}

\def \mint{-\hskip -1.06 em \int}
\def \re {\mathbb R}

\def \svre {\hbox{\svfilt R}}

\def \wh {\widehat}
\def \wt {\widetilde}
\def\Upsilon{\wt{\cal I}}

\def \T{\mathrm{Tr}\,}
\def \DM{\mathcal D\mathcal M_\infty}

\date{}

\begin{document}

\title{Shape derivatives for minima of integral functionals}

\author{Guy BOUCHITT\'E$^*$, \  Ilaria FRAGAL\`A$^\sharp$, \ Ilaria LUCARDESI$^\sharp$
\\
{\small $*$ Laboratoire IMATH, Universit\'e de Toulon et du Var,
+83957 La Garde Cedex (France)}
\\
{\small $\sharp$ Dipartimento di Matematica, Politecnico di Milano,
Piazza L. da Vinci, 20133 Milano (Italy)}}

\maketitle

\begin{abstract}
For $\Omega$ varying among
open bounded sets in $\re ^n$, we consider shape functionals $J (\Omega)$ defined as the infimum over a Sobolev space of an integral energy of the kind
$\int _\Omega[ f (\nabla u) + g (u) ]$, under Dirichlet or Neumann conditions on $\partial \Omega$.
Under fairly weak assumptions on the integrands $f$ and $g$, we prove that, when a given domain $\Omega$ is deformed into a one-parameter family of domains $\Omega _\e$ through an initial velocity field $V\in W ^ {1, \infty} (\re ^n, \re ^n)$, the corresponding shape derivative of $J$ at $\Omega$ in the direction of $V$ exists.  Under some further regularity assumptions, we show that the shape derivative can be represented as a boundary integral depending linearly on the normal component of $V$ on $\partial \Omega$. Our approach to obtain the shape derivative is new, and it is based on the joint use of Convex Analysis and Gamma-convergence techniques. It allows to deduce, as a companion result, optimality conditions in the form of conservation laws.
\end{abstract}

{\small {\it Keywords}: shape functionals, infimum problems, domain derivative, duality. }

{\small {\it MSC2010}: 49Q10, 49K10, 49M29, 49J45.}

\section{Introduction}
The theory of shape derivatives is a widely studied topic, with many applications in variational problems and optimal design.  Its origin can be traced back to the first half of the last century, with the pioneering work by Hadamard \cite{H}, followed by Schiffer and Garabedian \cite{S,GS}.
Afterwords, some important advances came in the seventies by C\'ea, Murat, and Simon \cite{C,MS,Si}. From the nineties forth, the many contributions given by different authors are witness of a renewed interest, partly motivated by the impulse given by the development of the field of numerical analysis in the research of optimal shapes.  We refer to  the recent monograph \cite{HP} by Henrot and Pierre as a reference text (see also the books \cite{DZ,SZ}), and, without any attempt of completeness, to the representative works \cite{Bu,DP,DZ1,G,GM,NP}.

In this paper we deal with the shape derivative of functionals which are obtained by minimizing a classical integral of the Calculus of Variations, under Dirichlet or Neumann conditions. Namely, we consider the functionals of domain defined, for
$\Omega$ varying among open bounded subsets of  $\re^n$, by \begin{eqnarray}
&J_D(\Omega):=\displaystyle{-\inf\left\{ \int_\Omega\big [ f(\nabla u) + g(u) \big ]\, dx \ :\ u\in W^{1,p}_0(\Omega)\right\}}
&\label{J_D}
\\ \noalign{\medskip}
&J_N(\Omega) := \displaystyle{-\inf\left\{ \int_\Omega\big [ f(\nabla u) + g(u) \big ]\, dx \ :\ u\in W^{1,p}(\Omega)\right\}\,.}
\label{J_N}
\end{eqnarray}
Here
 $f: \re ^n \to \re$ and $g: \re \to \re$ are continuous and convex integrands, which satisfy growth conditions of order $p$ and $q$ respectively, specified later on. We point out that we have put a minus sign in front of the infima in (\ref{J_D})-(\ref{J_N})
 just for a matter of convenience; 
 indeed by this way, in the typical case when $f(0) = g (0)= 0$, we deal with positive shape functionals.

In the sequel, the notation $J (\Omega)$ is adopted for brevity in all the statements which apply indistinctively in the Dirichlet and Neumann cases.

Given a vector field $V$ in
$W ^ {1, \infty} (\re ^n, \re ^n)$, we consider the one-parameter family of domains which are obtained as deformations of $\Omega$ with $V$ as initial velocity, that is we set
\begin{equation}\label{omegae}
\Omega_\e:=\Big \{ x + \e V ( x) \ :\ x \in \Omega \Big \}\,, \qquad \e > 0\ .
\end{equation}

By definition, the shape derivative of $J$ at $\Omega$ in direction $V$, if it exists, is given by the limit
\begin{equation}\label{deri1}
J'(\Omega,V):=\displaystyle{\lim_{\e\to 0^+ } \frac{J(\Omega_\e)-J(\Omega)}{\e} }\,.
\end{equation}

The approach we adopt in order to study the shape derivative (\ref{deri1}) is different from the one usually employed in the literature, and seems to have a twofold interest:
on one hand it allows to obtain the shape derivative for more general integrands $f$ and $g$;
on the other hand, it leads to establish conservation laws for solutions to problems (\ref{J_D})-(\ref{J_N}).

Before describing the results, let us briefly recall the habitual approach to the computation of $J' (\Omega, V)$, in order to enlighten the difference of perspective.  Classically, the object of study in theory of shape derivatives is the differentiability  at $\e
 = 0^+$ of functions of the form
\begin{equation}\label{integral_functional}
I(\e):=\int_{\Omega_\e} \phi(\e,x)\,dx\ ,
\end{equation}
being $\Omega_\e$ a one-parameter family of deformations of $\Omega$, as in (\ref{omegae}). As a special case of (\ref{integral_functional}), one can deal with shape derivatives for minima of integral functionals: namely, letting $u_\e$ be a solution to the infimum problem $J(\Omega_\e)$ and choosing
\begin{equation}\label{integrand}
\phi(\e,x):=- \big [ f(\nabla u_\e(x)) +  g(u_\e(x))]\ ,
\end{equation}
there holds $J(\Omega_\e)=I(\e)$.
The differentiability at $\e = 0 ^+$ of the map $I (\e)$, along with the formula for its right derivative, is proved in
 \cite{HP} assuming suitable regularity hypotheses on the integrand $\phi$.
Thus,
in order to deal with shape functionals like (\ref{J_D}) or (\ref{J_N}), one has to check that
the function $\phi(\e, x)$ defined by (\ref{integrand})  fulfills the afore mentioned regularity hypotheses.
This check has to be done case by case, according to the choice of $f$ and $g$, and requires in particular to compute, by means of
the Euler-Lagrange equation satisfied by $u_\e$,
 the derivative
\begin{equation}\label{u'} u' := \frac{d}{d \e} u _\e \big |_{\e = 0 ^+}\, .\end{equation}

Subsequently, further regularity assumptions
must be imposed in order to obtain structure theorems and representation results for shape derivatives, which lead to express them as boundary integrals over $\partial \Omega$.  We refer to \cite{HP} for a detailed presentation.

\medskip

Adopting a completely different point of view, in this paper we propose a new approach based on the combined use of Convex Analysis and Gamma-convergence. In particular, we heavily exploit the dual formulation of $J(\Omega)$,  which in the Dirichlet and Neumann cases reads respectively
\begin{eqnarray}
& \displaystyle{J_D^*(\Omega)\!=\! \inf\left\{ \! \int_\Omega \! [f^\ast(\sigma) + g^\ast(\div \sigma)]  dx \, :\, \sigma \in L^{p'}(\Omega;\re^n)\,, \ \div \sigma \in L ^ {q'} (\Omega) \right\}}\,& \label{J_D*}
\\ \noalign{\smallskip}
 &\displaystyle{J_N^*(\Omega)\!= \!\inf\left\{\!  \int_\Omega \![ f^\ast(\sigma) + g^\ast(\div \sigma)] dx \, :\, \sigma \in L^{p'}(\Omega;\re^n), \, \div \sigma \in L ^ {q'} (\Omega)\,, \, \sigma \cdot n = 0 \hbox{ on } \partial \Omega \!\right\}\ \ }& \label{J_N*}
\end{eqnarray}
where $f^\ast$ and $g^\ast$ denote the Fenchel conjugates of $f$ and $g$, $p'$ and $q'$ are the conjugate exponents of $p$ and $q$, while $\sigma \cdot n$ is the normal trace of $\sigma$ on $\partial \Omega$ intended in the sense of distributions (see Lemma \ref{lemma_opt_cond}).

Our strategy consists in 
giving lower and upper bounds for the quotient

$$
    q_\e (V):=\frac{ J(\Omega _\e) - J(\Omega) }{\e}
\,,
$$
by exploiting in particular the fact that $J(\Omega_\e)$ or $J(\Omega)$ can be rewritten in dual form according to (\ref{J_D*})-(\ref{J_N*}).
Such bounds read respectively
\begin{equation}\label{q-}
\liminf _{\e \to 0 ^+} q _\e (V) \geq \inf _{\sigma \in \mathcal S ^*} \sup _{u \in \mathcal S} \int_\Omega A(u, \sigma): DV\,dx \end{equation}
and
\begin{equation}\label{q+}
\limsup _{\e \to 0 ^+} q _\e(V) \leq \sup _{u \in \mathcal S} \inf _{\sigma \in \mathcal S ^*} \int_\Omega A(u, \sigma): DV\, dx\,
\end{equation}
where $\mathcal S$ and $\mathcal S ^*$ denote the set of solutions to $J (\Omega)$ and $J ^* (\Omega)$, $A(u, \sigma)$ is the  tensor defined on the product space  $\mathcal S \times \mathcal S ^*$ by
\begin{equation}\label{tensA}
A(u, \sigma) := \nabla  u \otimes \sigma  - [f(\nabla u)+g(  u)]\,I
\end{equation}
(being $I$ the identity matrix), and $A(u, \sigma):DV$ denotes the Euclidean scalar product of the two matrices.
Since the  inf-sup at the r.h.s. of (\ref{q-}) is larger than or equal to the sup-inf at the r.h.s. of (\ref{q+}),
we conclude that they agree, and that the limit as $\e \to 0 ^+$ of $q _\e (V)$, namely the shape derivative $J' (\Omega, V)$, exists.  Moreover,  denoting by $(u ^\star, \sigma ^\star) \in \mathcal S \times \mathcal S ^*$ an element where the value of the sup-inf or inf-sup is attained, there holds
$$J' (\Omega, V) = \int_\Omega A(u^\star, \sigma^\star): DV\, dx\, ,$$
see Theorem \ref{thm_firstder}.

In general, since the pair $(u ^ \star, \sigma ^ \star)$ in the above representation formula may depend on $V$, one cannot assert that $J' (\Omega, V)$ is a linear form in $V$ (for a more detailed discussion in this respect, see Remark \ref{remnonlin} and Example \ref{exBV}).

Nonetheless, as soon as the function $f$, the solutions $ u$ to $J (\Omega)$ and the boundary $\partial \Omega$
satisfy some fairly weak regularity conditions,
the shape derivative can be recast as a boundary integral depending linearly on $V$
(see eq. (\ref{J'bis}) in Theorem \ref{cor_firstder});
we stress that such conditions do not include the existence of a boundary trace for $\nabla u$,
as it is typical in the classical formulas for shape derivatives.
In turn, if $\nabla u$ turns out to be regular enough,
the shape derivative can be rewritten under the customary form of a boundary integral
depending only on the normal component of $V$ on $\partial \Omega$
(see eq. (\ref{thesis_traces}) in Theorem \ref{cor_firstder}).


\medskip
As a by-product of the above described bounds for $q _\e (V)$, we obtain a result which seems to have an autonomous interest, namely the validity of optimality conditions in the form of conservation laws for the variational problems under study. Actually, by making horizontal variations (somewhat in the same spirit of \cite{FG}), that is by exploiting the vanishing of $q _\e (V)$ for all $V \in \mathcal D (\Omega, \re ^n)$, we infer that suitable tensors of the type (\ref{tensA}) turn out to be divergence-free, see Corollary \ref{optiA}.
In particular, in case $f$ is Gateaux-differentiable except
at most at the origin, the outcome is  simply the following distributional equality

\begin{equation}\label{divnulla}
\div \Big ( \nabla u \otimes \nabla f ( \nabla  u)  - [f(\nabla  u)+g(  u)]\,I \Big ) = 0 \qquad \forall u \in {\mathcal S}\, .
\end{equation}
Notice that in the scalar case $n=1$, this relation reduces to the classical conservation law
$${u} '   f' ( {u} ')  - [f( {u} ')+g( u)] = c\, ,$$
which is obtained as a first integral of the Euler-Lagrange equations for smooth Lagrangians, see {\it e.g.} \cite[Corollaire 2.1.6]{Dac}.
Surprisingly the higher dimensional version (\ref{divnulla}) seems not to be widely known, and we could find it in the literature just in the regular case (see \cite{francfort}).

Let us emphasize that, in our approach, we never make use of the derivative function $u'$ in (\ref{u'}), so that we can by-pass the problem of investigating the validity of the Euler-Lagrange equation for minimizers (about which we refer to the recent papers \cite{BoCeMa, DeMa} and references therein). Thus we may deal also with integral functionals whose minima satisfy just a variational inequality. As a significant example in this respect, let us consider the
functional $J _ D (\Omega)$ in (\ref{J_D}), when $\Omega$ is a planar domain, and the convex integrands $f$ and $g$ are defined by
\begin{equation}\label{IFB}
f(y):= \begin{cases} \frac{|y| ^ 2}{2} + \frac{1}{2} & \hbox { if } |y| \geq 1
\\ |y| & \hbox{ if } |y| < 1\, \end{cases} \,,  \qquad g (y ) := - \lambda y \ \ \  (\lambda \in \re).
\end{equation}
This problem, studied in our previous works \cite{ABFL, BFLS}, arises in
the shape optimization of thin rods in pure torsion regime, settled on the bar cross-section.
Due to the lackness of regularity of $f$ at the origin, the optimality condition satisfied by elements of $\mathcal S$ is not an Euler-Lagrange equation, but merely a variational inequality (see \cite[Proposition 3.1]{ABFL}).
In this situation, the classical approach fails, whereas our existence and representation results can be applied to compute the shape derivative.  Let us mention that a few references are available in the literature about shape derivatives for problems governed by variational inequalities, see \cite{HL, NSZ, SZ}.

\medskip
Finally, let us comment on some different kinds of extensions of our approach, which
go beyond the scopes of the present manuscript.

The kind of integrands  under consideration could be generalized to the case $h (x, u, \nabla u)$, with a convex dependence on $(u, \nabla u)$, and a measurable dependence on $x$.  We limited ourselves to the case $f (\nabla u) + g (u)$ just because the main concern of the paper is to illustrate the new approach without too many technicalities (especially the dependence on $x$ would make some of the proofs heavier).

Another more challenging extension concerns the case of non-convex integrands. In fact, we believe that the convexity assumption made on $g$ in this paper can be relaxed. However, this is not a purely technical variant. In fact, a major difficulty arises in this case, namely finding an appropriate {\it alter ego} of the classical dual problem: a possible attempt in this direction, which deserves further investigation, is the new duality method introduced in \cite{ABDM}.

Finally, our approach can be successfully extended to compute second order shape derivatives.
Clearly, this would require to make stronger regularity assumptions, but the principle of matching upper and lower bounds continues to work perfectly well.

\medskip
The paper is organized as follows.

Section \ref{prel} is devoted to the preliminary material: we fix the main notation, the standing assumptions, and the  basic lemmata concerning the functionals under study.

In Section \ref{sec_first} we state our main results, which are proved in Section \ref{sec_second}.

Section \ref{secapp} is an appendix where we gather, for the sake of safe-completeness, some
auxiliary results which are exploited at some point in the paper.

\bigskip

{\bf Acknowledgments.} We are grateful to A.~Henrot for pointing out some related references, and to P.~Bousquet for some interesting discussions on the topics of Remark  \ref{remLip}.

\section{Preliminaries}\label{prel}

\subsection{Notation}\label{not}

In the following $\Omega$ denotes an open bounded connected subset of $\re^n$. 

We recall that throughout the paper the notation $J(\Omega)$ is adopted each time it can be intended indistinctly as in $(\ref{J_D})$ or  as in $(\ref{J_N})$.
Similarly,  $J ^*(\Omega)$ is meant either as in
$(\ref{J_D*})$ or as in $(\ref{J_N*})$.

Only when required, we shall distinguish between the Dirichlet and the Neumann cases, indicated respectively as (D) and (N) in the sequel.

For brevity, we denote by $W(\Omega)$ the domain of admissible functions for $J (\Omega)$ (namely $W ^ {1,p} _0 (\Omega)$ in case (D) and
$W ^ {1,p} (\Omega)$ in case (N)), and by $X(\Omega; \re ^n)$ the domains of admissible vector fields for $J^* (\Omega)$ (namely the space of $L ^ {p'}$ vector fields with divergence in $L ^{q'}$ in case (D), with the additional condition $\sigma \cdot n = 0$ on the boundary in case (N)).

Moreover, we define the subsets $\mathcal S$ of $W(\Omega)$ and $\mathcal S^*$ of $X(\Omega; \re ^n)$ by
$$
\mathcal S:= \Big \{ \hbox{ solutions to  } J (\Omega) \Big \} \qquad \hbox{ and } \qquad
\mathcal S^*:= \Big \{ \hbox{ solutions to  } J ^*(\Omega) \Big \}\,.$$

Given $V\in W^{1,\infty} (\mathbb R^n;\mathbb R^n)$ and $\e>0$, we denote by $\Psi_\e$ and $\Omega _\e$ the bi-Lipschitz diffeomorphism of $\mathbb R ^n$ and the corresponding image of an open set $\Omega \subset \mathbb R ^n$ defined respectively by
$$
\Psi_\e(x):=x+\e V(x) \qquad \hbox{ and } \qquad \Omega _\e:= \Psi _\e (\Omega)\,.$$

We adopt the convention of repeated indices. Given two vectors $a,b$ in $\re^n$  and two matrices $A$ and $B$ in $\re^{n \times n}$,
we use the standard notation $a \cdot  b$ and $A:B$ to denote their Euclidean scalar products, namely
$a  \cdot  b =  a_i b _i$ and $A:B= A_{ij} B_{ij}$.
We denote by $a\otimes b$ the tensor product of $a$ and $b$, namely the matrix $(a\otimes b)_{ij}:=a_i b_j$, and by $I$ the identity matrix.
We denote by $A^{-1}$ and $A^{T}$ the inverse and the transpose matrices of $A$, and by $A^{-T}$ the transposition of the inverse of $A$.
Given a tensor field $A\in C^1(\re^n;\re^{n\times n})$, by $\div A$ we mean its divergence with respect to lines, namely
$(\div A)_i:= \partial_j A_{ij}$.

Given $1\leq p \leq +\infty$ we denote by $p'$ its conjugate exponent, defined as usual by the equality $1/p+1/p'=1$.

We denote by $\mathcal D(\Omega)$ the space of $\mathcal C^{\infty}$ functions having compact support contained into $\Omega$, and by ${\rm Lip}(\Omega)$ the space of Lipschitz functions on $\Omega$.

In the integrals, unless otherwise indicated, integration is made with respect to the $n$-dimensional Lebesgue measure. Furthermore,
in all the situations when no confusion may arise, we omit to indicate the integration variable. We denote by $\mint$ the integral mean.

Whenever we consider $L ^p$-spaces over $\Omega$ and over $\partial \Omega$, they are intended the former
with respect to the $n$-dimensional Lebesgue measure over $\Omega$, and the latter with respect to the $ (n-1)$-dimensional Hausdorff measure over $\partial \Omega$.

We now recall some fact about weak notions of traces, assuming that $\Omega$ has a Lipschitz boundary, with unit outer normal $n$. 

If $v\in W^{1,p}(\Omega)$, we denote by $\T(v)$ its trace on $\partial \Omega$, which can be characterized via the divergence theorem by
\begin{equation}\label{BV_div}
\int_{\partial \Omega} \T(v) \varphi\,n_i\, d\mathcal H^{n-1}= \int_\Omega v\, \big [
\partial_i \varphi +  \varphi \, (D_i v) \big ]\qquad \forall \varphi \in C ^ 1 (\overline \Omega)\,.
\end{equation}
The trace operator
$v \mapsto \T(v)$ is linear and bounded from $W^{1,p}(\Omega)$ to $L^p(\partial \Omega)$.
Moreover, $\T(v)$ can be computed as
$$\T(v)(x_0)=\lim_{r,\rho\to 0^+}\mint_{C^-_{r,\rho}(x_0)} v \qquad \hbox{ for $\mathcal H^{n-1}$-a.e. $x_0\in \partial \Omega$,}
$$
where $C^-_{r,\rho}(x_0)$ denotes the inner cylindrical neighborhood
\begin{equation}\label{cilindro}
C^-_{r,\rho}(x_0):=\{y\in \Omega\ :\ y=x-t n(x_0)\,,\ x\in B_\rho(x_0)\cap \partial \Omega\,,\ t\in (0,r)\}\,.
\end{equation}
In particular, in case $v\in W^{1,p}(\Omega)\cap C^0(\ov{\Omega})$, $\T(v)$ coincides with the restriction of $v$ to $\partial \Omega$.

We remark that a similar notion of trace extends to functions $v\in BV(\Omega)$, and in this case $v \mapsto \T(v)$ defines a bounded linear operator from $BV(\Omega)$ to $L^1(\partial \Omega)$, see  \cite{AFP}.

Finally, let us recall the definition of normal trace for vector fields in the class
$$
\DM(\Omega; \re^n):=\big \{\Psi \in L^\infty(\Omega;\mathbb R^n)\ :\ \div \Psi \in \mathcal M(\Omega) \big \}\,,
$$
where $\mathcal M(\Omega)$ denotes the space of Radon measures over $\Omega$, {\it cf.}\ \cite{A,CF}.
For every $\Psi \in \DM (\Omega; \re ^n)$, there exists a unique function $[\Psi\cdot n]_{\partial \Omega}\in L^\infty(\partial \Omega)$ such that
\begin{equation}\label{gen_div}
\int_{\partial \Omega} [\Psi\cdot n]_{\partial \Omega}\,\varphi\,d{\mathcal  H^{n-1}}=\int_\Omega \big [ \Psi\cdot \nabla \varphi +\varphi\,\div\Psi \big ] \qquad \forall \varphi \in C ^1 (\overline \Omega)\ .
\end{equation}
Equipped with the norm $\|\Psi\|_{\DM }:=\|\Psi\|_{\infty} + |\div \Psi|(\Omega)$, $\DM (\Omega;\mathbb R^n)$ is a Banach space, and
the normal trace operator
$
\Psi\mapsto [\Psi\cdot n]_{\partial \Omega}
$ from $\DM (\Omega; \re ^n)$ to $L^\infty(\partial \Omega)$ is linear and bounded. Moreover, we recall from  \cite[Proposition 2.2]{Aup} that, if $\partial \Omega$ is piecewise $C^1$,  $[\Psi\cdot n]_{\partial \Omega}$ can be computed as
\begin{equation}\label{pointwisenormalt}
 [\Psi\cdot n]_{\partial \Omega} (x_0) = \lim_{r,\rho\to 0^+}\mint_{C^-_{r,\rho}(x_0)} \Psi \cdot \tilde n \qquad \hbox{for $\mathcal H^{n-1}$-a.e. $x_0\in \partial \Omega$}\, ,
 \end{equation}
being $\tilde n$ the extension of $n$ to ${C^-_{r,\rho}(x_0)}$ defined by
\begin{equation}\label{estensione}
\tilde n (y) := n (x) \qquad \hbox { if } y = x - t n (x_0)\,.
\end{equation}

In particular, in case $\Psi\in \DM (\Omega; \re ^n) \cap C^0(\ov\Omega; \re ^n)$, the normal trace operator applied to $\Psi$ agrees with the normal component of the pointwise trace:
$$
[\Psi\cdot n]_{\partial \Omega}(x_0)=\Psi(x_0)
\cdot n(x_0) \qquad \forall x_0 \in  \partial \Omega\ . $$

In the sequel, we also use the notation $\DM (\Omega; \re ^{n\times n})$  and $\DM  (\Omega)$ to denote respectively the class of tensors $A$ with rows in $\DM (\Omega; \re ^{n})$, and the class of scalar functions $\psi$ with $\psi I \in  \DM (\Omega; \re ^{n\times n})$. Accordingly, we indicate by
$[A\, n]_{\partial \Omega}$ and $[\psi\, n]_{\partial \Omega}$ the normal traces of $A$ and $\psi I$ intended row by row as in (\ref{gen_div}).

For a more detailed account of the theory of weak traces, we refer the reader to \cite{ACM,A,CF}.

\subsection{Standing assumptions}\label{SA}
Throughout the paper, we work under the following hypotheses, which will be referred to as standing assumptions:

\begin{itemize}
\item[(H1)]  $\Omega$ is an open bounded connected set;

\smallskip
\item[(H2)]  $V$ is a vector field in $W ^ {1, \infty} (\re ^n; \re ^n)$;
\smallskip
\item[(H3)]
$f:\re^n\to \re$ and $g:\re\to \re$ are convex, continuous functions such that $g(0)=0$ and
\begin{equation}\label{pq}
\begin{cases}
&\alpha ( |z|^p - 1)  \leq  f(z) \leq \beta (|z|^p + 1) \qquad \forall z \in \re ^n
\\ \noalign{\medskip}
& \gamma ( |v|^p - 1)  \leq g(v) \leq \delta (|v|^{q}+1) \qquad \  \forall v \in \re\ .
\end{cases}
\end{equation}
\end{itemize}

Here $\alpha,\beta,\gamma$ are positive constants, while the exponents $p$, $q$ are assumed to satisfy
 $$
 1<p<+\infty \, ,
  \qquad
\qquad \begin{cases}
 q= p ^* : = \frac{np}{n-p} & \hbox{  if } p<n  \\
 1< q <+\infty  & \hbox{  if } p\geq n \, .\end{cases}
 $$

We remark that choosing the exponent $q$ as above makes the upper bound for $g$ in (\ref{pq})  less restrictive than the  one  asked for $f$; concerning the lower bound for $g$ in (\ref{pq}),
in the Dirichlet case it  can be relaxed to
\begin{equation}\label{weak-below}
- \gamma (|v|+1) \leq g(v) \qquad \forall v \in \re\, .
\end{equation}
Notice that a positive constant $\gamma$ such that (\ref{weak-below}) holds true exists for any real valued continuous convex function $g$, as it admits  an affine minorant.

When further assumptions on $f$ and $g$ are needed, they will be specified in each statement.

\subsection{Basic lemmata on integral functionals}

\begin{lemma}\label{I_fg}
Under the standing assumptions on $f$ and $g$, let $I _ f$ and $I _g$ be defined respectively on $L^p(\Omega;\mathbb R^n)$ and $L^q(\Omega)$ by
\begin{equation}\label{defI}
I_f(z):=\int_\Omega f(z)\qquad \hbox{ and } \qquad  I _ g (v):= \int _\Omega g (v)\, .
\end{equation}
Then:
\begin{itemize}

\item[(i)] the functionals $I_{f} (z)$ and $I _g (u)$ are  convex, finite, strongly continuous and weakly l.s.c.\ respectively on $L^p(\Omega;\mathbb R^n)$ and $L^q(\Omega) $;
\item[(ii)] the functional
$I_{f} (\nabla u) +I _g (u)$ is convex, finite, weakly coercive and weakly l.s.c.\ on $W(\Omega)$;

\item[(iii)] the sets ${\mathcal S}$ and ${\mathcal S} ^*$ of solutions to $J (\Omega) $ and $J ^ * (\Omega)$ are nonempty.
\end{itemize}
\end{lemma}
\proof
(i) Since $f$ and $g$ are convex and continuous they admit an affine minorant, namely there exist $a,b\in \re^n$ and $\alpha,\beta \in \re$ such that, for every $z\in \re^n$, $u\in \re$,
\begin{equation}\label{af_min}
a+ b\cdot z \leq f(z)\ ,\ \ \alpha + \beta u \leq g(u)\ .
\end{equation}
Recalling that $q>1$, condition (\ref{af_min}), together with the growth assumption from above, implies that $f$ and $g$ satisfy
\begin{equation}\label{abs_growth}
|f(z)|\leq C(|z|^p + 1)\ ,\ \ |g(u)|\leq C'(|u|^q + 1)\ ,
\end{equation}
for some positive constants $C$ and $C'$, and for every $z\in \re^n$, $u\in \re$. Exploiting (\ref{af_min}), (\ref{abs_growth}) and the properties of continuity and convexity of $f$ and $g$, we infer that $I_f$ and $I_g$ are both convex, finite, strongly continuous and weakly lower semicontinuous on the functional spaces $L^p(\Omega;\re^n)$ and $L^q(\Omega)$ respectively (see \cite[Corollary 6.51 and Theorem 6.54]{FL}).

(ii) By combining statement (i) with the standing assumptions on the exponent $q$, it is easily checked that the fuctional $I_f (\nabla u)+I_g(u)$ is finite, convex and weakly lower semicontinuous on $W(\Omega)$. Moreover, the growth condition from below on $f$ and $g$ ensures the weak coercivity.

(iii) In view of (ii), the existence of a solution to $J(\Omega)$  follows from the direct method of the Calculus of Variations. Finally, the existence of at least one solution to $J^*(\Omega)$ follows from the equality $J(\Omega)=J^*(\Omega)$ (see Lemma {\ref{lemma_opt_cond}}) and the duality Proposition \ref{prop_duality}.

\qed

\begin{lemma}\label{lemma_opt_cond}
Under the standing assumptions, the functionals defined by $(\ref{J_D})$ (resp.\ $(\ref{J_N})$) and $(\ref{J_D*})$ (resp.\ $(\ref{J_N*})$) coincide, namely  there holds
\begin{equation}\label{J=J*}
J(\Omega)=J^\ast(\Omega)\ .
\end{equation}
Moreover, if ${u\in W(\Omega)}$ and $\sigma \in X (\Omega; \re ^n)$,  there holds the following equivalence:
$$
(i) \left\{
\begin{array}{lll}
{u}\in \mathcal S
\\ {\sigma} \in \mathcal S^* \end{array}
\right. \quad \Longleftrightarrow\quad (ii) \left\{ \begin{array}{lll} {\sigma}\in \partial f(\nabla {u})\ \ a.e.\ in\ \Omega \\ \div {\sigma} \in \partial g({u})\ \ a.e.\ in\ \Omega\ . \end{array} \right.
$$
\end{lemma}

\begin{remark}\label{rem_Fenchel}
{\rm
For a better understanding of the previous lemma let us recall that, if
 $Y$ is a normed vector space with topological dual $Y^*$, and $F:Y\to \re$ is a proper function, the sub-gradient of $F$ at a point $y\in Y$ admits the following characterizations:
\begin{equation}\label{subgrad}
y^*\in \partial F(y)\quad \Longleftrightarrow\quad y\in \partial F^*(y) \quad \Longleftrightarrow\quad F(y) + F^*(y^*) = \langle y^*,y\rangle_{Y^*,Y}\ ,
\end{equation}
where the latter condition is usually called {\it Fenchel equality}.

In view of (\ref{subgrad}), we infer that conditions (i) and (ii) in Lemma \ref{lemma_opt_cond} are also equivalent to
$$
\left\{ \begin{array}{lll} \nabla {u}\in \partial f^*({\sigma})\ \ a.e.\ in\ \Omega \\ u \in \partial g^*(\div {\sigma})\ \ a.e.\ in\ \Omega\ . \end{array} \right.
$$
}
\end{remark}

{\it Proof of Lemma \ref{lemma_opt_cond}.}
In order to prove the equality (\ref{J=J*}), we are going to apply a standard Convex Analysis Lemma, which is enclosed in the Appendix
for convenience of the reader ({\it cf.}\ Proposition \ref{prop_duality}).
Introducing the Banach spaces $Y:= W(\Omega)$, $ Z:=L^p(\Omega;\mathbb R^n)\times L^q(\Omega)$, the function $\Phi: Y \to \re$ identically zero, the function $\Psi:Z\to \mathbb R$ defined by $
\Psi(z,u):=I_f(z)+I_g(u)$, and the linear operator $A:Y \to Z$ defined by $ A(u):=(\nabla u,u)$,
we can rewrite the shape functional $J(\Omega)$ as
$$
J(\Omega)=-\inf_{u\in Y}\left\{\Psi(Au) + \Phi (u) \right\}\ .
$$
From Lemma \ref{I_fg} (i), we infer that $\Psi$ is convex, finite and sequentially continuous on $Z$.
Finally, if $u _0 \equiv 0$, it holds $\Phi(u_0)<+\infty$ and $\Psi$ is continuous at $A(u_0)$.
Then Proposition \ref{prop_duality} applies and gives
\begin{equation}\label{equality}
J (\Omega)=\inf_{(\sigma,\tau) \in Z^*} \left\{ \Psi^*(\sigma,\tau) + \Phi^*(-A^*(\sigma,\tau))  \right\}\ .
\end{equation}

Let us compute the Fenchel conjugates $\Psi^*$, $\Phi^*$ and the adjoint operator $A^*$.

By Propositions \ref{prop_subdif} and \ref{Ih*2}, for every $(\sigma,\tau)\in Z^*= L^{p'}(\Omega;\mathbb R^n)\times L^{q'}(\Omega)$ there holds
$$
\Psi^*(\sigma,\tau)=(I_f)^*(\sigma) + (I_g)^*(\tau)=I_{f^*}(\sigma) + I_{g^*}(\tau)\ .
$$

Since $\Phi\equiv 0$, its Fenchel conjugate $\Phi ^*$ is $0$ at $0$ and $+\infty$ otherwise.

As an element of $Y^*$, $A^*(\sigma,\tau)$ is characterized by its action on the elements of $Y$: since
$$
\langle A^*(\sigma,\tau),u\rangle_{Y^*,Y}
= \langle (\sigma,\tau), A(u)\rangle_{Z^*,Z}=\langle \tau, u\rangle_{L^{q'},L^q} +
\langle \sigma, \nabla u\rangle_{L^{p'},L^p}
= \int_\Omega \tau\,u + \sigma \cdot \nabla u \ ,
$$
we infer that $A^*(\sigma,\tau)=0$ if and only if $\tau=\div \sigma$ (with the additional condition $\sigma \cdot n=0$ in the sense of distributions on the boundary in case (N)).
Hence the r.h.s.\ of (\ref{equality}) agrees with $J _ D ^* (\Omega)$ in case $(D)$, and with $J _ N ^* (\Omega)$
in case (N).
We conclude that (\ref{J=J*}) holds.

It remains to check the equivalence between conditions (i) and (ii). By Proposition \ref{prop_duality}, condition (i) holds true if and only if  $(\div \sigma,\sigma)\in \partial \Psi(A(u))$.
In view of Proposition \ref{prop_subdif}, there holds
$
\partial \Psi(A({u}))= \partial I_f(\nabla u)\times \partial I_g (u)
$,
and hence, by Proposition \ref{Ih*2}, we have  $(\div \sigma, \sigma)\in \partial \Psi(A( u))$ if and only if condition (ii) holds true.

\qed

We now endow $W (\Omega)$ and $X(\Omega; \re ^n)$ respectively with the following convergence, which in both cases will be simply called  {\it weak convergence}:
\begin{eqnarray}
& u_k\stackrel{W^{1,p}}{\rightharpoonup} u_0\ , \label{G_conv}
\\
&\sigma_k \stackrel{L^{p'}}{\rightharpoonup} \sigma_0\ ,\qquad \div\sigma_k\stackrel{L^{q'}}{\rightharpoonup} \div\sigma_0\ . \label{G_conv*}
\end{eqnarray}

\begin{lemma}\label{compactness}
 Under the standing assumptions, the sets $\mathcal S$ and $\mathcal S ^*$ are weakly compact respectively in $W(\Omega)$ and $X (\Omega; \re ^n)$.
\end{lemma}
\proof

Let $u_k$ be a sequence of elements in  $\mathcal S$. By the coercivity statement in Lemma \ref{I_fg} (ii), the sequence is bounded in $W(\Omega)$, hence it admits a subsequence which converges in the weak $W^{1,p}$-topology  to some $u\in W(\Omega)$. By the l.s.c. statement in Lemma \ref{I_fg} (ii),  we infer that also the limit function $u$ belongs to $\mathcal S$.

\medskip

In view of Lemma \ref{lemma_opt_cond}, we can rewrite the set $\mathcal S^*$ as
\begin{equation}\label{S*}
\mathcal S^*=\{\sigma \in X (\Omega; \re ^n) \ :\ \sigma \in \partial f(u_0)\  a.e. \, , \  \div \sigma \in \partial g(u_0)\  a.e. \}\ ,
\end{equation}
with $u_0$ arbitrarily chosen in $ \mathcal S$. By  Lemma \ref{I_fg} (i),
the functionals $I _f$ and $I _g$ defined in (\ref{defI})
are convex and strongly continuous on $L ^ p (\Omega; \re ^n)$ and $L ^ q (\Omega; \re ^n)$. Then we can apply Proposition \ref{clarke} to infer that the sets
$\partial f(\nabla u_0)$ and $\partial g(u_0)$ are weakly compact respectively in
$L^{p'}(\Omega;\mathbb R^n)$ and in $L^{q'}(\Omega)$.
Hence, exploiting the characterization (\ref{S*}) and taking into account that the constraint $\tau = \div \sigma$ is weakly closed, we conclude that $\mathcal S^*$ is weakly compact in $X (\Omega; \re ^n)$.

\qed

\section{Main results}\label{sec_first}

We begin by introducing the following tensor which will play a crucial role in the sequel.

\begin{definition}\label{A1} {\rm  For any $(u, \sigma) \in \mathcal S \times\mathcal S^*$, we set
$$
A (u,  {\sigma} ):= \nabla  u \otimes \sigma - [f(\nabla u)+g( u)]\,I  \ .
$$ }
\end{definition}

\begin{remark}\label{A2}
{\rm

\begin{itemize}
\item[(i)] Thanks to the growth conditions (\ref{pq}) satisfied by $f$ and $g$, it is easy to check that $A( u,  \sigma)\in L^1(\Omega;\re^{n\times n})$.

\item[(ii)] By using the Fenchel equality satisfied by $u$ and $\sigma$ ({\it cf.} Lemma \ref{lemma_opt_cond} (ii)), the tensor $A(u,\sigma)$ can be rewritten as
$$A (u,  {\sigma} )= \nabla  u \otimes \sigma + [f^*(\sigma)+g^*(\div \sigma) - \nabla u\cdot \sigma - u\div \sigma]\,I  \ .
$$
\item[(iii)] In case $f$  is Gateaux differentiable except at most in the origin, the optimality condition $ \sigma \in \partial f(\nabla  u)$ holding for all $(u, \sigma) \in \mathcal S \times \mathcal S^*$ determines uniquely $ \sigma$ (as $\nabla f (\nabla  u)$) in the set $\{\nabla u\neq 0\}$. Therefore in this case the tensor $A(u, \sigma)$ turns out to be independent of $\sigma$, and as such it will be simply denoted by $A (u)$. Namely, when $f$  is Gateaux differentiable except at most in the origin, for any $u \in \mathcal S$ we set
\begin{equation}\label{A(ubar)}
A ( u ) := \nabla u \otimes \nabla f (\nabla u) - [f(\nabla u)+g( u)]\,I   \,.
\end{equation}
This tensor $A(u)$ is sometimes called {\em energy-momentum tensor} ({\it cf.}\ \cite{francfort}).
\end{itemize}
}
\end{remark}

\medskip

We are now in a position to state our main results.

\begin{theorem}\label{thm_firstder} {\bf (existence of the shape derivative)}

Under the standing assumptions, the shape derivative of the functional $J(\cdot)$ at $\Omega$ in direction $V$ defined according to $(\ref{deri1})$ exists. Actually, for every $V \in W ^ {1, \infty} (\re ^n; \re ^n)$, the following inf-sup and sup-inf agree and are equal to $J' (\Omega, V)$:
\begin{equation}\label{J'}
J'(\Omega,V)=  \sup _{ u \in \mathcal S} \inf _{\sigma \in \mathcal S ^*}  \int_\Omega A (u, \sigma):DV =   \inf _{\sigma \in \mathcal S ^*}  \sup _{ u \in \mathcal S} \int_\Omega A (u, \sigma):DV \,.
\end{equation}
Moreover, there exists a saddle point $(u^\star, \sigma ^\star) \in \mathcal S \times \mathcal S ^*$ at which the inf-sup and sup-inf above are attained.
\end{theorem}

\medskip
\begin{remark}\label{remnonlin} {\rm  In general, equality (\ref{J'}) does not allow to conclude that $V \mapsto J ' ( \Omega, V)$ is a linear form, since {\it a priori}
the pair $(u ^ \star, \sigma ^ \star)$ depends on $V$. However, the linearity of the shape derivative in $V$ can be asserted in one of the following situations:

\smallskip
-- when both primal and dual problems have a unique solution (as in this case both $\mathcal S$ and $\mathcal S^*$ are singletons);

\smallskip
-- when the primal problem has a unique solution $\overline u$, and $f$ is Gateaux differentiable except at most at the origin
(as in this case $\mathcal S$ is a singleton, and the tensor $A$ depends only on $\overline u$ ).

\smallskip
In particular, in the latter case we are going to see that, under some additional regularity assumptions on $\ov u$ and $\partial \Omega$, the shape derivative can also be recast as a boundary integral depending linearly on the normal component of $V$ on the boundary, see Theorem \ref{cor_firstder} below.
}
\end{remark}

\medskip
As a complement to Remark \ref{remnonlin}, we exhibit below an example of shape functional whose derivative $J' (\Omega, V)$ is {\it not} a linear form in $V$. We are aware that such functional does not respect the growth conditions (\ref{pq}), but we were unable to individuate an equally simple one-dimensional example fitting our standing assumptions.
On the other hand, the extension of our results to the limit case $p=1$ is likely possible but requires many additional technicalities so that it goes beyond the purpose of this work.

\begin{example}\label{exBV} {\rm
Let $J(\Omega)$ be the one-dimensional shape functional given on open sets $\Omega \subset \mathbb R$ by
$$
J (\Omega) := - \inf \Big \{  \int _{\Omega} [ \, |u'| + ( 1- u ) _+ \, ] \ \ :\ \  u \in W^{1, 1}_0 (\Omega) \Big \} \, ,
$$
where $( \ \cdot\   ) _+$ stands for the positive part.

We claim that, for every $a>0$ and for every deformation $V\in W ^{1, \infty} (\re)$, there holds:
\begin{eqnarray}
& J \big ( (0, a) \big )\  =  \ - \min \big \{ 2, a \} & \label{explicitj}
\\  \noalign{\medskip}
& J '\big ((0, 2), V \big )\  =  \  (V(0)- V(2))_+\,.  \label{ex_J'}
\end{eqnarray}
Notice that the validity of (\ref{ex_J'}) implies in particular that $J'\big ((0, 2), V \big )$ does not depend linearly on $V$. Indeed,
setting $m (a) := J \big ( (0, a) \big )$, there holds
$m' _-  (2) = -1 \neq  0 = m'  _+(2)$. Moreover, taking $V= V (x)$ as the sum of two deformations compactly supported respectively near $x=0$ and $x= 2$, we see from (\ref{ex_J'}) that $J '\big ((0, 2), V \big )$ also fails to be linear with respect to variations with disjoint supports.

Let us prove (\ref{explicitj}) and (\ref{ex_J'}). The equality (\ref{explicitj}) readily follows by applying Proposition \ref{prop_duality} (with $X=W^{1,1}_0(\Omega)$ and $Y=L^1(\Omega)$). Indeed we get for $m(a)$ the dual formulation
\begin{equation*}
m (a) = \inf \left\{\int_0^a g^*(\sigma')\ :\ \sigma\in {\rm Lip}(0,a)\,,\ |\sigma|\leq 1\ \hbox{a.e.}\right\}
\end{equation*}
being the Fenchel conjugate $g^*$ of $g$ given by $g^*(\xi)=\xi$ in $[-1,0]$ and $+\infty$ elsewhere.
It is then easy to check that
\begin{align}
m(a) &= \inf \left\{\int_0^a \sigma' \ :\ \sigma\in {\rm Lip}(0,a)\ ,\ |\sigma|\leq 1\ , \ -1\leq \sigma'\leq 0\ \hbox{a.e.}\right\} \label{ex_dual}
\\ \nonumber
& = \inf \left\{\sigma(a)-\sigma(0)\ :\ \sigma\in {\rm Lip}(0,a)\ ,\ |\sigma|\leq 1\ ,\ -1\leq \sigma'\leq 0\ \hbox{a.e.}\right\} = -\min \{2,a\}\ .
\end{align}

The equality (\ref{ex_J'}) is a straightforward consequence of (\ref{explicitj}) and the fact that, if $\Omega$ is an interval of length $a$, the value of $J(\Omega)$ depends only on $a$.

Finally, let us discuss the validity of representation formula (\ref{J'}) for $\Omega=(0,2)$; more precisely, let us show how
the representation formula (\ref{J'}) allows to recover the equality  (\ref{ex_J'}) .

We observe that, for  $a=2$, the dual problem (\ref{ex_dual}) has a unique solution, namely $\mathcal S^*$ is a singleton and reads
\begin{equation}\label{ex_S*}
\mathcal S^*=\left\{\overline{\sigma}=-x-1\right\}\ .
\end{equation}
On the other hand, we notice that for $a<2$ the unique solution of the primal problem is the constant zero, whereas for $a\geq 2$ no solution exists, that is,
$\mathcal S $ is empty.
However, following De Giorgi (see for instance \cite{ET}, Chap V, Sec 2.3), we may relax the boundary condition in the primal problem as follows:
\begin{equation*}\label{ex_relaxed}
m(a)=-\inf\left\{ \int_0^a [|u'| + (1-u)_+]\,dx + |u(0)| + |u(a)| \ :\ u\in W^{1,1} (0,a)   \right\}.
\end{equation*}
Indeed, if $\tilde m(a)$ denotes the right hand side infimum above, then obviously $\tilde m(a) \ge m(a)$. The  opposite inequality
can be derived by observing that, for every  $u\in W^{1,1}(0,a)$ and every competitor $\sigma$ for \eqref{ex_dual}, one has
 $$\int_0^a [|u'| + (1-u)_+]\,dx + |u(0)| + |u(a)|\ \ge\  \int_0^a [u'\sigma + (u-1) \sigma']\, dx  + |u(0)| + |u(a)|\ \ge \ - \int_0^a \sigma'.$$

For   $a=2$, the associated set of ``relaxed'' solutions turns out to be
\begin{equation}\label{ex_S}
\widetilde{\mathcal S}=\{u_\lambda\equiv \lambda\ :\ \lambda \in [0,1]\}\ .
\end{equation}

By combining (\ref{ex_S*}) and (\ref{ex_S}), we infer that the family of tensors introduced in Definition \ref{A1} depends only on the parameter $\lambda\in [0,1]$. An easy computation leads to
\begin{align*}
A(u_\lambda,\overline{\sigma}) & = u_\lambda'\overline{\sigma} - [f(u_\lambda')+g(u_\lambda)]= \lambda -1\ .
\end{align*}
Eventually we find that, by applying the min-max formula (\ref{J'}) (in which we substitute $\mathcal S$ by  $\widetilde{\mathcal S}$),
we recast the shape derivative in (\ref{ex_J'}):
$$
\sup_{u\in \widetilde{\mathcal S}}\inf_{\sigma\in \mathcal S^*} \int_0^2 A(u,\sigma)\,V' = \sup_{\lambda\in [0,1]} \int_0^2 A(u_\lambda,\overline{\sigma})V'  =\sup_{\lambda\in [0,1]} (V(2)-V(0)) \, (\lambda-1) = (V(0)-V(2))_+\,.
$$
}

\end{example}

\medskip

Related to the conservation law (\ref{divnulla}) announced in Section 1, a nice consequence of Theorem \ref{thm_firstder} is the following:

\begin{corollary} \label{optiA}  {\bf (conservation laws)}

Under the standing assumptions, there holds:

\smallskip
(i) For every $\ov u \in \mathcal  S$,  there exists  $\wh \sigma = \wh \sigma (\ov u) \in\mathcal S ^*$ such that
\begin{equation}\label{divfree1}
\div A (\ov u, \wh \sigma)=0  \qquad \hbox{ in } \mathcal D'(\Omega; \re ^n)\, .
\end{equation}
In particular, in case $f$ is Gateaux differentiable except at most at the origin, for every $\ov u \in \mathcal S$ there holds
\begin{equation}\label{divfree2}
\div  A (\overline u) = 0 \qquad \hbox { in } {\mathcal D} ' (\Omega; \re ^n)\,.
\end{equation}

\smallskip
(ii) For every $\ov  \sigma \in \mathcal S ^*$, there exists $\wh u = \wh u (\ov \sigma)\in\mathcal S$ such that
\begin{equation}\label{divfree3}
\div A (\wh u, \ov \sigma)=0  \qquad \hbox{ in } \mathcal D'(\Omega; \re ^n)\,.
\end{equation}
\end{corollary}

\medskip

Thanks to equality (\ref{divfree2}) in Corollary \ref{optiA}, when $\overline u$ satisfies suitable regularity assumptions as specified below, the associated tensor $A (\overline u)$ turns out to admit a normal trace on the boundary, which can also be characterized in terms of the energy density. Thus, as a consequence of Theorem \ref{thm_firstder} and  Corollary \ref{optiA},
we obtain:
\begin{theorem}\label{cor_firstder} {\bf (shape derivative as a linear form on the boundary) }

Under the standing assumptions, suppose in addition that
$f$ is Gateaux differentiable except at most at the origin, that $\Omega$ has a Lipschitz boundary with unit outward normal $n$, and assume that problem $J (\Omega)$ admits a unique solution $\overline u$, with $\overline u \in {\rm Lip} (\Omega)$. Then $A(\u)$ belongs to $\DM(\Omega;\re^{n\times n})$ and the shape derivative given by $(\ref{J'})$ can be recast as the linear form
\begin{equation}\label{J'bis}
J'(\Omega,V)= \int_{\partial \Omega} [A({\overline u})\, n]_{\partial \Omega} \cdot V \, d {\mathcal H} ^ {n-1} \,.
\end{equation}
Further, if we assume that $\partial \Omega$ is piecewise $C^1$, that $\nabla \overline u \in BV (\Omega; \re ^n)$, and that there exists $\overline\sigma \in \mathcal S^*\cap BV(\Omega;\re^n)$,
then it holds
\begin{equation}\label{thesis_traces}
J'(\Omega,V)= \begin {cases}
\displaystyle{\int_{\partial \Omega} \T \big (  f^*(\overline \sigma)     \big ) (V \cdot n )\, d {\mathcal H} ^ {n-1} }
&
\hbox{ in case $(D)$}
\\
\noalign{\smallskip}
\displaystyle{\int_{\partial \Omega} \T \big ( f(  \nabla \overline u ) + g (\overline u )  \big )  ( V \cdot n  ) \, d {\mathcal H} ^ {n-1} }
 & \hbox{ in case $(N)$} \,.
\end{cases}
\end{equation}
\end{theorem}

\medskip
\begin{remark}\label{remLip}{\rm The Lipschitz regularity of solutions to classical problems in the Calculus of Variations is a delicate matter which is object of current investigation. In particular, it is out of our scopes here to discuss the conditions on $f$ and $g$ which yield Lipschitz solutions to $J (\Omega)$ as assumed in 
Theorem \ref{cor_firstder}.  Without any attempt of completeness, we refer the interested reader to the papers \cite{Ce, FiTr, MaTr, Mi} for both local and global regularity results.

Also the gradient regularity of solutions (and more generally their higher differentiability)
is a delicate matter, which has been object of investigation under different smoothness and growth conditions
on the Lagrangian. We refer for instance to the recent papers \cite{CaKrPa,CeMa} and references therein.
}\end{remark}

\medskip
\begin{example} {\rm (i) As a model example, we can recover from Theorem \ref{cor_firstder} the shape derivative of the classical torsion functional
$$J (\Omega)=- \min \Big \{ \int _\Omega \big ( \frac{1}{2} |\nabla u| ^ 2 - u\big )  \ :\ u \in H ^ 1 _0 (\Omega)\Big \} \,.$$
Indeed, since in this case $f ^ * (\overline \sigma) = \frac{1}{2} |\overline \sigma | ^ 2$, and $\overline \sigma = \nabla \overline u$, we get
\begin{equation}\label{extor}
J' (\Omega, V) =   \int _{\partial \Omega} \frac{1}{2} |\nabla \overline u | ^ 2 \, (V\cdot n)\, d \mathcal H ^ { n-1}\ ,
\end{equation}
where $\overline u$ is the solution to the classical torsion problem in $\Omega$, namely the unique solution in $H ^ 1 _0 (\Omega)$ to the equation $- \Delta \overline u = 1$ in $\Omega$.
In view of (\ref{extor}), the boundary value problem satisfied by $\overline u$ becomes overdetermined with a constant Neumann boundary condition if $\Omega$ is a minimizer for the shape functional $J ( \cdot)$ under a volume constraint.

\medskip

(ii) A variant of the above case not covered by the classical literature is when $f$ and $g$ are taken as in (\ref{IFB}), for instance with $\lambda = 1$. Denoting by $J (\Omega)$ the corresponding Dirichlet shape functional, we get
\begin{equation}\label{extorgen}
J' (\Omega, V) =  \int _{\partial \Omega } \frac{1}{2} \big ( | \nabla \overline u | ^ 2 - 1)_+ \, (V\cdot n)\,  d \mathcal H ^ { n-1}\ ,
\end{equation}
where $\overline u$ is any solution to the primal problem and $t_+$ denotes the positive part of real number $t$. Indeed, for such a solution $\overline u$,
it holds $\overline \sigma = \nabla \overline u$ on the subset $|\nabla\overline u|\ge 1$ whereas
$f ^ * (\overline \sigma) =\frac{1}{2} (|\overline \sigma | ^ 2- 1)_+$.  

We point out that all solutions $\overline u$ verify a variational inequality  (see \cite[Proposition 3.1]{ABFL}) and therefore the problem of minimizing $J (\cdot)$ under a volume constraint leads in a natural way to consider overdetermined variational inequalities, which is an interesting field to our knowledge unexplored.
}

\end{example}


\section{Proofs}\label{sec_second}

In order to prove the results stated in the previous Section, we are going to analyze thoroughly
the asymptotic behavior as $\e\to 0^+$ of the following sequence of differential quotients, which can be expressed in two equivalent forms thanks to Lemma \ref{lemma_opt_cond}:
$$
q_\e(V)= \frac{J^*(\Omega_\e)-J(\Omega)}{\e}=\frac{J(\Omega_\e)-J^*(\Omega)}{\e}\, .
$$

More precisely, we proceed along the following outline: first we rewrite $J(\Omega_\e)$ and $J^*(\Omega_\e)$ as infimum problems for  integral functionals set over the fixed domain $\Omega$ (Lemma \ref{rewriting}) and we study the asymptotic behavior of their solutions (Proposition \ref{lemma_Gamma}); as a consequence, we provide a lower bound and an upper bound for $q_\e (V)$ (Propositions \ref{thm_lowerbound2} and \ref{thm_upperbound}); afterwards, by exploiting these bounds, we prove Theorem \ref{thm_firstder} and, finally, all the other results stated in Section \ref{sec_first}.

\bigskip

For every $\e>0$, let us introduce the functionals $E_\e$ and $H_\e$ defined respectively on $W (\Omega)$ and $X(\Omega; \re ^n)$ by
\begin{eqnarray}
        E_\e(u)&:=& \int_\Omega [f(D\Psi_\e^{-T}\nabla u) + g(u)]\,|\det D\Psi_\e|\ , \label{Ee}\\
H_\e(\sigma) &:=& \int_\Omega [f^*(|\det D\Psi_\e|^{-1}D\Psi_\e \sigma) + g^*(|\det D\Psi_\e|^{-1}\div \sigma)]\,|\det D\Psi_\e|\ . \label{Ee*}
\end{eqnarray}

For brevity, in the sequel we will also employ the notation
$$\beta_\e:=|\det D\Psi_\e|\ , \qquad
f_\e(x,z)  := f(D\Psi_\e^{-T}z)\beta_\e\ ,
\qquad
g_\e(x,v)  := g(v)\beta_\e \label{g_e}\ .
$$

We recall that, for $\e>0$ small enough and for a.e. $x\in \Omega$, the coefficient $\beta_\e$ is strictly positive and the matrix $D\Psi_\e$ is invertible.

We are now in a position to rewrite the infimum problems under study on the fixed domain $\Omega$:


\begin{lemma}\label{rewriting}
Under the standing assumptions, for every $\e>0$ there holds
\begin{align}
J(\Omega_\e)&=-\inf\big \{ E_\e(u)\ :\ u\in W(\Omega)\big \}\ ,\label{rewrite1}
\\ \noalign{\medskip}
J^*(\Omega_\e)&=\ \ \, \inf\big \{ H_\e(\sigma)\ :\ \sigma \in X(\Omega; \re ^n) \big \}\ .\label{rewrite2}
\end{align}
\end{lemma}

\begin{proof}
Let $\e>0$ be fixed.
Functions  $\tilde u\in W(\Omega_\e)$ are in  1-1 correspondance with functions $u\in W(\Omega)$ through the equality $\tilde u=u\circ\Psi_\e^{-1}$ in $\Omega_\e$; moreover, via change of variables, there holds
$$
\int_{\Omega_\e} [f(\nabla \tilde u) + g(\tilde u)] = \int_\Omega  [f _\e (x,\nabla u) + g _\e (x,u)]\ .
$$
Passing to the infimum over $\tilde{u}\in W(\Omega_\e)$ in the l.h.s.\ and over $u
\in W (\Omega)$ at the r.h.s., yields (\ref{rewrite1}).

\medskip

By arguing in the same way as already done in the proof of Lemma \ref{lemma_opt_cond}, we obtain that the dual form of  $J (\Omega _\e)$  is given 
by
\begin{equation}\label{dualeps}
J ^ * (\Omega _\e) = \inf\left\{ \int_\Omega [f_\e^*(x,\sigma) + g^*_\e(x,\div\sigma)]\ :\ \sigma\in X(\Omega; \re ^n)\right\}\ ,
\end{equation}
where  $f_\e^*$ and $g_\e^*$ denote the Fenchel conjugates of $f$ and $g$,  performed with respect to the second variable. Their computation gives:
\begin{align*}
f_\e^*(x,z^*) & = \sup_{z\in \svre^n}\left\{z\cdot z^* - f(D\Psi_\e^{-T}z)\beta_\e\right\}
= \beta_\e\, f^*(\beta_\e^{-1}D\Psi_\e\,z^*)\ ,
\\
g_\e^*(x,v^*) & = \sup_{v\in \svre}\left\{v\,v^* - g(v)\beta_\e\right\}= \beta_\e\, g^*(\beta_\e^{-1} v^*)\ .
\end{align*}
Inserting these expressions into (\ref{dualeps}) and  keeping in mind definition (\ref{Ee*}), we obtain (\ref{rewrite2}).
\end{proof}

Now, as a key step,
we establish the $\Gamma$-convergence of the functionals $E _\e$ and $H _\e$
to the limit functionals defined on $W (\Omega)$ and $X(\Omega; \re ^n)$ respectively by
\begin{eqnarray}
           E(u)&:=& \int_\Omega [f(\nabla u) + g(u)]\ , \label{E}\\
H(\sigma) &:=& \int_\Omega [f^*(\sigma) + g^*(\div \sigma)] \ .\label{E*}
\end{eqnarray}
We recall that $W (\Omega)$ and $X(\Omega; \re ^n)$ are endowed with the weak convergence defined in
(\ref{G_conv})-(\ref{G_conv*}).

\begin{proposition}\label{lemma_Gamma} {\bf ($\Gamma$-convergence)}

\smallskip

(i) On the space $W(\Omega)$ endowed with the weak convergence, the sequence $E_\e$ in $(\ref{Ee})$ is equicoercive and, as $\e\to 0$, it $\Gamma$-converges to the functional $E$ in $(\ref{E})$.
In particular, every sequence $u_\e$ such that $u_\e \in \mathrm{Argmin}(E_\e)$ admits a subsequence which converges weakly in $W (\Omega)$ to some $u_0\in \mathrm{Argmin}(E)$.

\smallskip
(ii) On the space $X(\Omega; \re ^n)$ endowed with the weak convergence, the sequence $H _\e$ in $(\ref{Ee*})$ is equicoercive and, as $\e\to 0$, it $\Gamma$-converges to the functional $H$ in $(\ref{E*})$.
In particular, every sequence $\sigma_\e$ such that $\sigma_\e \in \mathrm{Argmin}(H_\e)$ admits a subsequence  which converges weakly in $X(\Omega; \re ^n)$ to some $\sigma_0\in \mathrm{Argmin}(H)$.

\end{proposition}

\begin{proof}
(i) The equicoercivity of the family of functionals $E_\e$ can be easily obtained by exploiting the growth assumptions (\ref{pq}) on $f$ and $g$, the uniform boundedness from below of the positive coefficients $\beta_\e$, and the uniform control on the $L^\infty$ norm of the tensors $D\Psi_\e^{-T}$. Let us prove the $\Gamma$-convergence statement for $E _\e$. By definition of $\Gamma$-convergence, we have to show that the so-called $\Gamma$-liminf and $\Gamma$-limsup inequalities hold, namely:
\begin{equation}\label{G_liminf}
\inf \left\{\liminf E_\e(u_\e)\ :\ u_\e\stackrel{W^{1,p}}{\rightharpoonup} u\right\}\geq E(u)\ ,
\end{equation}
\begin{equation}\label{G_limsup}
\inf\left\{\limsup E_\e(u_\e)\ :\ u_\e \stackrel{W^{1,p}}{\rightharpoonup}u\right\}\leq E(u)\ .
\end{equation}

\medskip

Let us prove (\ref{G_liminf}). Consider an arbitrary sequence $u_\e$ which converges weakly to $u$ in $W(\Omega)$.  We observe that the sequence
$D\Psi_\e^{-T}\nabla u_\e$ converges to $\nabla u$  weakly in $L^p(\Omega;\mathbb R^n)$, and that  $\beta_\e$ converges to $1$ uniformly in $\Omega$, except in a negligible set.
Hence, exploiting the weak lower semicontinuity of $I_f$ and $I_g$ in $L^p(\Omega;\mathbb R^n)$ and $L^q(\Omega)$ respectively ({\it cf.} Lemma \ref{I_fg} (i)), we infer that
$$
E(u)\leq \liminf_\e \int_\Omega f(D\Psi_\e^{-T}\nabla u_\e)\beta_\e +  \liminf_\e \int_\Omega g(u_\e)\beta_\e\,dx \leq \liminf_\e E_\e(u_\e)\ ,
$$
which implies (\ref{G_liminf}).

\medskip

Let us prove (\ref{G_limsup}). For every fixed $u\in W(\Omega)$ we have to find a recovery sequence, namely a sequence $u_\e$ which converges weakly to $u$ in $W(\Omega)$ and satisfies
 \begin{equation}\label{G_limsup2}
E(u)\geq \limsup_\e E_\e(u_\e)\ .
\end{equation}
We claim that the sequence $u_\e\equiv u$ for every $\e>0$ satisfies (\ref{G_limsup2}). Indeed, since
$D\Psi_\e^{-T}\nabla u$ converges stronlgly to $\nabla u$ in $L^p(\Omega;\mathbb R^n)$,
by exploiting Lemma \ref{I_fg} (i) we obtain :
$$
E(u)=\lim_{\e \to 0}  \int_\Omega f(D\Psi_\e^{-T}\nabla u)\beta_\e +  \lim_{\e \to 0}  \int_\Omega g(u)\beta_\e \geq \limsup_{\e} E_\e(u)\ .
$$

\medskip

Finally, the compactness of a sequence of minimizers $u_\e$ follows from the equicoercivity of the family $E_\e$, while the optimality of a cluster point $u_0$ is a well-known consequence of $\Gamma$-convergence (see {\it e.g.} \cite[Corollary 7.20]{DM}).

\bigskip

(ii) The equicoercivity of the sequence $H_\e$ can be easily obtained by exploiting the uniform boundedness from below of the positive coefficients $\beta_\e$, the uniform control on the $L^\infty$ norm of the tensors $D\Psi_\e^{-T}$, and the following growth conditions, holding for some positive constants $a,b$ as a consequence of the standing assumption (\ref{pq}):
\begin{align*}
f^*(z^*) & \geq a(|z^*|^{p'}-1)\qquad \forall z^*\in \mathbb R^n\, ,
\\
g^*(v^*) & \geq b(|v^*|^{q'}-1)\, \qquad \forall v^*\in \mathbb R\, .
\end{align*}
Let us prove the $\Gamma$-convergence statement for $H _\e$.  We observe that
the $\Gamma$-convergence of the functionals $E_\e$ to $E$ proved at item (i) above
can be proved in the same way also on the product space
\begin{equation}\label{space_product}
\big \{ (u, \nabla u) \, :\, u \in L ^ q (\Omega)\, , \nabla u \in L ^ p (\Omega; \re ^n) \big \}
\end{equation}
endowed with the product of the weak  $L^q$ and weak $L ^p$ convergences.
Moreover, such $\Gamma$-convergence can be strengthened into a Mosco-convergence, since we have exhibited a recovery
sequence which converges in the strong topology. We recall that the sequence $E_\e$ Mosco-converges to $E$ if and only if (\ref{G_liminf}) holds as such, and (\ref{G_limsup}) holds replacing therein the weak $W^{1,p}$ convergence by the strong one (see \cite{Mo}).
Since the Mosco-convergence is stable when passing to the Fenchel conjugates (see  \cite[Theorem 1.3]{AzAtWe}), we deduce that the functionals $E _\e ^*$ Mosco-converge (and hence $\Gamma$-converge)
to the functional $E^*$. We conclude by noticing  that
the dual of the product space in (\ref{space_product})
(endowed with the product of the weak  $L^{q'}$ and weak $L ^{p'}$ convergences)
is precisely $X (\Omega; \re ^n)$ (endowed with the weak convergence in (\ref{G_conv*})), and on such space $H _\e$ and $H$ agree respectively with the  Fenchel conjugates $E_\e ^*$ and $E^*$.

Finally, as for (i), the compactness of a sequence of minimizers $\sigma_\e$ follows from the equicoercivity of the family $H_\e$, and the optimality of a cluster point $\sigma_0$ is again a consequence of $\Gamma$-convergence.

\end{proof}

\begin{proposition}\label{thm_lowerbound2} {\bf (lower bound)}

Under the standing assumptions, it holds
$$
\liminf_{\e \to 0^+} q_\e(V) \geq \inf_{\sigma \in \mathcal S^*} \sup_{u\in \mathcal S} \int_\Omega A(u,\sigma):DV\ .
$$
\end{proposition}
\proof

We are done if we show the existence of $\sigma_0\in \mathcal S^*$ such that
\begin{equation}\label{lowerbound}
\liminf_{\e \to 0^+} q_\e(V) \geq  \sup_{u\in \mathcal S} \int_\Omega A(u,\sigma_0):DV \,.
\end{equation}
To that aim we observe that, if $H_\e$ and $E$ are the functionals defined respectively in (\ref{Ee*}) and (\ref{E}), by Lemma \ref{lemma_opt_cond} and Lemma \ref{rewriting}
there holds
\begin{equation}\label{lower0}
q_\e(V)=\frac{J^*(\Omega_\e)-J(\Omega)}{\e}=\frac{\inf H_\e + \inf E}{\e}\ .
\end{equation}
Let $\sigma _\e \in {\rm Argmin} (H_\e)$ and
$u\in \mathrm{Argmin}(E)=\mathcal S$.
In view of (\ref{lower0}),  $q_\e(V)$ reads
$$
q_\e(V)=\frac{1}{\e}
\left[
\int_\Omega [f^*(\beta_\e^{-1} D\Psi_\e\sigma_\e) + g^*(\beta_\e^{-1} \div \sigma_\e) ]\beta_\e + \int_\Omega [f(\nabla u) + g(u)]
\right]\ .
$$
Since the coefficient $\beta_\e$ is (strictly) positive a.e., by applying the Fenchel inequality we obtain
\begin{equation}\label{lower1}
q_\e(V)\geq \frac{1}{\e}
\left[
\int_\Omega [(D\Psi_\e\sigma_\e)\cdot \nabla u + u \div \sigma_\e ] - \int_\Omega [f(\nabla u) + g(u)] (\beta_\e-1)
\right]\ .
\end{equation}
Recalling that $D\Psi_\e=I+\e DV$, an integration by parts gives
\begin{equation}\label{lower2}
\int_\Omega [(D\Psi_\e\sigma_\e)\cdot \nabla u + u \div \sigma_\e ] = \e \int_\Omega  (DV \sigma_\e) \cdot \nabla u\ .
\end{equation}
By combining (\ref{lower1}) and (\ref{lower2}), we obtain
\begin{equation}\label{lower3}
q_\e(V)\geq \int_\Omega  [(DV \sigma_\e) \cdot \nabla u] - \int_\Omega [f(\nabla u) + g(u)] \frac{(\beta_\e-1)}{\e}\ .
\end{equation}
In order to show (\ref{lowerbound}), we now want to pass to the limit as $\e \to 0 ^+$ in the r.h.s.\ of (\ref{lower3}).

By Proposition \ref{lemma_Gamma} (ii), up to subsequences there holds
\begin{equation}\label{sigma0}
\sigma_\e \stackrel{L^{p'}}{\rightharpoonup} \sigma_0
\end{equation}
for some $\sigma_0\in \mathrm{Argmin}(H)= \mathcal S^*$. (Notice that {\it a priori} $\sigma_0$, as well as $\sigma_\e$, depend on $V$).

On the other hand, we observe that $\beta_\e= 1 + \e \div V + \e^2 m_\e$, where $m_\e \in L ^ \infty (\Omega)$ satisfy $\sup_\e\|m_\e\|_{L^\infty(\Omega)}\leq C$ for some positive constant $C$. Therefore, a.e. in $\Omega$
\begin{equation}\label{ratio}
\frac{\beta_\e -1}{\e}\rightarrow \div V\qquad \mathrm{uniformly}\ .
\end{equation}

Thanks to (\ref{sigma0}) and (\ref{ratio}), by passing to the limit as $\e\to 0^+$ in (\ref{lower3}), we conclude that
$$
\liminf_{\e \to 0^+}q_\e(V)\geq \int_\Omega  [(DV \sigma_0) \cdot \nabla u] - \int_\Omega [f(\nabla u) + g(u)] \div V = \int_\Omega A(u,\sigma_0):DV\ .
$$
Finally, by the arbitrariness of $u\in \mathcal S$, we obtain (\ref{lowerbound}).

\qed

\begin{proposition}\label{thm_upperbound} {\bf (upper bound)}

Under the standing assumptions, it holds
$$
\limsup_{\e\to 0^+}q_\e(V)\leq \sup _{ u \in S} \inf_{\sigma \in \mathcal S^*} \int_\Omega A(u,\sigma):DV\ .
$$
\end{proposition}

\proof
We are done if we show the existence of $u_0\in \mathcal S$ such that
\begin{equation}\label{upperbound}
\limsup_{\e\to 0^+}q_\e(V)\leq \inf_{\sigma \in \mathcal S^*} \int_\Omega A(u_0,\sigma):DV\ .
\end{equation}

To that aim we observe that, if $E_\e$ and $H$ are the functionals defined respectively in (\ref{Ee}) and (\ref{E*}),
by Lemma \ref{lemma_opt_cond} and Lemma \ref{rewriting} there holds
\begin{equation}\label{q_eps}
q_\e(V)=\frac{J(\Omega_\e)-J^*(\Omega)}{\e}=\frac{-\inf E_\e - \inf H}{\e}\, .
\end{equation}
Let
$u _\e \in {\rm Argmin} (E_\e)$ and $\sigma \in \mathrm{Argmin}(H)=\mathcal S^*$.
In view of (\ref{q_eps}), $q _\e (V)$ reads
\begin{equation}\label{q_eps2}
q_\e(V)=\frac{1}{\e} \left[  - \int_\Omega\left[  f(D\Psi_\e^{-T}\nabla u_\e) + g(u_\e)  \right]\,\beta_\e - \int_\Omega [f^*(\sigma) + g^*(\div \sigma)]  \right]\ .
\end{equation}

We observe that $\beta_\e$ and $D\Psi_\e^{-1}$ admit the following asymptotic expansions in terms of $\e$:
$$
\beta_\e=1+\e \div V + \e^2 m_\e\ ,\ \ D\Psi_\e^{-1}=I-\e DV + \e^2 M_\e\ ,
$$
for some $m_\e\in L ^ \infty(\mathbb R)$ and $M_\e\in L ^ \infty(\mathbb R^n;\mathbb R^{n\times n})$.
Thus, if we apply the Fenchel inequality and we set for brevity
\begin{align*}
a_\e & := m_\e \sigma -\div V\, DV\,\sigma + M_\e \sigma -\e m_\e DV\sigma + \e \div V\,M_\e \sigma + \e^2 m_\e M_\e \sigma\ ,
\\
\alpha_\e & := m_\e \div \sigma\ ,
\end{align*}
we obtain
\begin{align}
 \int_\Omega\left[  f(D\psi_\e^{-T}\nabla u_\e) + g(u_\e)  \right]\,\beta_\e  & \geq \int_\Omega \left[(D\Psi_\e^{-1} \sigma) \cdot \nabla u_\e + \div \sigma u_\e - f^*(\sigma) - g^*(\div \sigma) \right]\,\beta_\e \notag
\\
 & =  \int_\Omega [\sigma \cdot \nabla u_\e + u_\e\, \div \sigma - f^*(\sigma) - g^*(\div\sigma) ] + \notag
 \\ & +  \e \int_\Omega [( \sigma \cdot \nabla u_\e + u_\e\,\div \sigma - f^*(\sigma) - g^*(\div\sigma))\div V - ( DV  \sigma) \cdot \nabla u_\e ] +\notag
 \\ & +  \e^2 \int_\Omega\left[ a_\e \cdot \nabla u_\e + \alpha_\e \, u_\e -m_\e (f^*(\sigma) +  g^*(\div\sigma))\right] \, .\label{estimate}
\end{align}
By combining (\ref{q_eps2}) and (\ref{estimate}), and recalling that\
$
\int_\Omega \sigma \cdot \nabla u_\e + u_\e\, \div \sigma=0\,,
$
we infer
\begin{equation}\label{q_eps3}
q_\e(V) \leq\int_\Omega [f^*(\sigma) + g^*(\div \sigma)]\div V +\int_\Omega [( DV \, \sigma) \cdot \nabla u_\e -( \sigma \cdot \nabla u_\e + u_\e\,\div  \sigma)\div V ] - \e C_\e\ ,
\end{equation}
where
$$
C_\e:=\int_\Omega\left[ a_\e \cdot \nabla u_\e + \alpha_\e \, u_\e -m_\e (f^*(\sigma) +  g^*(\div\sigma))\right]\ .
$$
In order to show (\ref{upperbound}), we now want to pass to the limit as $\e \to 0 ^+$ in the r.h.s.\ of (\ref{q_eps3}).
By Proposition \ref{lemma_Gamma} (i), up to subsequences there holds
\begin{equation}\label{u0}
u_\e \stackrel{W^{1,p}(\Omega)}{\rightharpoonup} u_0
\end{equation}
for some $u_0\in \mathrm{Argmin}(E)= \mathcal S$ (notice that {\it a priori}  $u_0$, as well as $u_\e$, depend on $V$).

On the other hand we remark that
$$
\sup_{\e}\|m_\e\|_{L^\infty(\Omega)}\leq C\ ,\ \ \sup_{\e}\|M_\e\|_{L^\infty(\Omega;\svre^{n\times n})}\leq C\ ,\ \
\sup_\e \|a_\e\|_{L^{p'}(\Omega;\svre^n)}\leq C\ ,\ \ \sup_\e \|\alpha_\e\|_{L^{q'}(\Omega)}\leq C\ ,
$$
which together with  (\ref{u0}) implies that the sequence $C_\e$ remains bounded as $\e$ goes to $0$. Then, by passing to the limit as $\e\to 0^+$ in (\ref{q_eps3}), we conclude that
\begin{equation}\label{q++}
\limsup_{\e\to 0^+}q_\e(V)\leq \int_\Omega [f^*(\sigma) + g^*(\div \sigma)]\div V +\int_\Omega [( DV  \sigma) \cdot \nabla u_0 -( \sigma \cdot \nabla u_0+ u_0\,\div  \sigma)\div V ]\ .
\end{equation}
Since $u_0\in \mathcal S$ and $\sigma\in \mathcal S^*$, recalling the optimality conditions (ii) in Lemma \ref{lemma_opt_cond}, we can rewrite the r.h.s.\ of (\ref{q++}) as
$$
\int_\Omega [ ( DV  \sigma) \cdot \nabla u_0 - (f(\nabla u_0) + g(u_0))\div V ]\,dx = \int_\Omega A(u_0, \sigma):DV\,.
$$
Finally, by the arbitrariness of ${\sigma\in \mathcal S^*}$, we obtain (\ref{upperbound}).

\qed

\bigskip

\bigskip
{\bf Proof of Theorem \ref{thm_firstder}.}

By Propositions \ref{thm_lowerbound2} and \ref{thm_upperbound}, there holds
\begin{equation}\label{inequalities}
\inf_{\sigma \in \mathcal S^*}\sup_{u \in \mathcal S} \int_\Omega A(u,\sigma):DV \leq \liminf_{\e\to 0^+} q_\e(V) \leq \limsup_{\e\to 0^+} q_\e(V)\leq \sup_{u \in \mathcal S} \inf_{\sigma \in \mathcal S^*}\int_\Omega A(u,\sigma):DV\ .
\end{equation}
Since the sup-inf at the r.h.s.\ of (\ref{inequalities}) is lower than or equal to the inf-sup at the l.h.s., we infer that all the inequalities in (\ref{inequalities}) are in fact equalities; in particular, the sequence $q_\e(V)$ converges as $\e\to 0^+$, and its limit provides the shape derivative $J'(\Omega,V)$, namely
$$
J'(\Omega,V)=\sup_{u \in \mathcal S}\inf_{\sigma \in \mathcal S^*} \int_\Omega A(u,\sigma):DV=\inf_{\sigma \in \mathcal S^*}\sup_{u \in \mathcal S} \int_\Omega A(u,\sigma):DV \ .
$$
Next, we observe that the functionals $\sigma \mapsto \int _\Omega A (u, \sigma):DV$ and $u \mapsto \int _\Omega A (u, \sigma):DV$ are linearly affine (see respectively  Definition \ref{A1} and Remark \ref{A2} (ii)), and hence weakly continuous respectively on $X(\Omega; \re ^n)$ and $W(\Omega)$.  Moreover, by Lemma \ref{compactness}, the sets $\mathcal S\subset W(\Omega)$ and $\mathcal S^*\subset X (\Omega; \re ^n)$ are weakly compact. Therefore, by Proposition \ref{infsup}, the
sup-inf or inf-sup above is attained at some optimal pair $( u^\star, \sigma^\star)\in \mathcal S\times \mathcal S^*$ (which {\it a priori} depends on $V$).

\qed

\bigskip
{\bf Proof of Corollary \ref{optiA}}.
Let $V$ be a deformation field in $\mathcal D (\Omega; \re ^n)$. Clearly, since $V$ is compactly supported into $\Omega$, for every $\e$ small enough the deformed domain $\Omega _\e$ in (\ref{omegae}) coincides with $\Omega$, so that
\begin{equation}\label{J'=0}
J'(\Omega,V)=0\ .
\end{equation}

In order to prove assertion (i), let us fix $\overline u\in \mathcal S$ and define $L_{\u}:\mathcal D(\Omega;\mathbb R^n)\times \mathcal S^*\to \re$ as
\begin{equation}\label{L1}
L_{\u}(V,\sigma):=-\int_\Omega A(\u,\sigma):DV\ .
\end{equation}
By the linearity with respect to $V$, (\ref{divfree1}) is equivalent to  showing the existence of $\wh \sigma\in \mathcal S^*$ such that
\begin{equation}\label{expression2}
\inf_{V\in  \mathcal D(\Omega;\svre^n)} L_{\u}(V,\wh \sigma)\ \ge\ 0\ .
\end{equation}
Since by Lemma \ref{compactness} the set  $\mathcal S^*$ is convex and weakly compact in $X(\Omega;\mathbb R^n)$, $L_{\u}(\cdot, \sigma)$ is convex, $L_{\u}(V,\cdot )$ is concave and weakly upper semicontinuous, Proposition \ref{infsup} applies and gives the existence of $\wh \sigma\in \mathcal S^*$ (depending on $\u$) such that
$$
\inf_{V\in \mathcal D(\Omega;\svre^n)}\sup_{\sigma\in \mathcal S^*} L_{\u}(V,\sigma)\ =\
\sup_{\sigma\in \mathcal S^*}\inf_{V\in \mathcal D} L_{\u}(V,\sigma)\ =\ \inf_{V\in  \mathcal D(\Omega;\svre^n)} L_{\u}(V,\wh \sigma)\ .
$$
Now the first term in previous equalities is non negative since by (\ref{J'=0}) and Theorem \ref{thm_firstder}, for every $V\in \mathcal D(\Omega;\mathbb R^n)$, it holds
$$
\sup_{\sigma \in \mathcal S^*} L_{\u}(V,\sigma)\ =\ -\inf_{\sigma \in \mathcal S^*}\int_\Omega A(\u,\sigma):DV \ge -J'(\Omega,V)=0\ .
$$
Thus we have proved (\ref{expression2}) and therefore assertion (i).

The proof of assertion (ii) is fully analogous: if $\ov  \sigma$ denotes a fixed element in $\mathcal S^*$ and (\ref{L1}) is replaced by the functional
 $L_{\ov \sigma}:\mathcal D(\Omega;\re^n)\times \mathcal S\to \re$ defined as
$$
L_{\ov \sigma}(V,u):=\int_\Omega A(u,\ov\sigma):DV\ ,
$$
by arguing  in the same way as done above to obtain (\ref{expression2}), one gets the inequality $\displaystyle\inf_{V\in  \mathcal D(\Omega;\svre^n)} L_{\ov \sigma}(V,\wh u)\ \ge\ 0$, which implies (\ref{divfree3}).

\qed

\bigskip

We now turn attention to the proof of Theorem \ref{cor_firstder}. To that aim, we need to state some preliminary facts about   boundary traces. They are collected in the next lemma, where we adopt the notation introduced in  Section \ref{not} for the traces in $BV$ and in $\DM $.

\begin{lemma}\label{prop_trace}
Given a domain $\Omega$  with piecewise $C^1$ boundary,  let $v$ and $\Psi$ be respectively a scalar function and a vector field defined on $\Omega$ which are both $L ^ \infty$ and $BV$. Denote by $C _{r, \rho} ^ -$ and $\tilde n$ the cylinder and the extension of the unit outer normal defined in $(\ref{cilindro})$ and $(\ref{estensione})$. Then the following equalities hold true at $\mathcal H ^ {n-1}$-a.e. $x_0 \in \partial \Omega$:
\begin{eqnarray}
& \T(v)(x_0)\,n(x_0) = [v \,n]_{\partial \Omega}(x_0)\,; & \label{duetracce}
 \\
\noalign{\medskip}
&\displaystyle{\lim_{r,\rho\to 0^+}\mint_{C^-_{r,\rho}(x_0)} \Big| \Psi(x)-\T(\Psi)(x_0) \Big|=0 \,;}
& \label{Leb_cyl2}
 \\
\noalign{\medskip}
 &\displaystyle{\lim_{r,\rho\to 0^+} \mint_{C^-_{r,\rho}(x_0)} \Big|  \Psi(x)\cdot \tilde n(x) - \T(\Psi)(x_0)\cdot n(x_0) \Big|=0}\ .
& \label{trace_tesi}
\end{eqnarray}

\end{lemma}
\begin{proof}
Let $v\in BV(\Omega)\cap L^\infty(\Omega)$. As an element of $BV(\Omega)$, $v$ has a trace $\T(v)\in L^1(\partial \Omega)$ and the product $\T(v)\,n$ is characterized in a functional way by (\ref{BV_div}).

On the other hand, since $v\in L^\infty(\Omega)$ and $Dv$ is a Radon measure over $\Omega$, we infer that $vI$ is an element of $\DM(\Omega;\re^{n\times n})$, namely $v\in \DM(\Omega)$. In particular, $v$ has a normal trace $[v \,n]_{\partial \Omega}\in L^\infty(\partial \Omega)$, which is characterized by (\ref{gen_div}).
%
By comparing the two characterizations (\ref{BV_div}) and (\ref{gen_div}) we infer that, for every test function $\varphi\in C^1(\ov\Omega)$, it holds
$$
\int_{\partial \Omega} \T(v)\varphi\,n \, d\mathcal H^{n-1}=\int_{\partial \Omega} [v\,n]_{\partial \Omega} \varphi\,d\mathcal H^{n-1}\ ,
$$
which implies the validity of (\ref{duetracce})  $\mathcal H ^ {n-1}$-a.e.\ on $\partial \Omega$.

The proof of (\ref{Leb_cyl2}) can be found in \cite[Section 5.3]{EG}.

Finally, in order to prove (\ref{trace_tesi}), we claim that, if
$x_0\in \partial \Omega$ is a Lebesgue point for  $n\in L^\infty(\partial \Omega)$, there holds
\begin{equation}\label{Leb_cyl1}
\lim_{r,\rho\to 0^+}\mint_{C^-_{r,\rho}(x_0)} | \tilde n(x)-n(x_0) |=0\ ,
\end{equation}

Once proved this claim, (\ref{trace_tesi}) follows easily. Indeed, by adding and subtracting suitable terms to the integrand  in (\ref{trace_tesi}), we obtain:
\begin{align}
&\mint_{C^-_{r,\rho}(x_0)} \Big|
 \Psi(x)\cdot \tilde n(x) - \T(\Psi)(x_0)\cdot n(x_0)\Big| \notag
\\
&\leq \mint_{C^-_{r,\rho}(x_0)} \Big|
\Psi(x)\cdot \tilde n(x) - \Psi(x)\cdot n(x_0)
 \Big| \notag
+ \mint_{C^-_{r,\rho}(x_0)} \Big|
\Psi(x)\cdot n(x_0) - \T(\Psi)(x_0)\cdot n(x_0)
\Big| \notag
\\
& \leq \|\Psi\|_{L^\infty(\Omega;\svre^n)} \mint_{C^-_{r,\rho}(x_0)} |\tilde{n}(x)-n(x_0)| + \mint_{C^-_{r,\rho}(x_0)} \Big| \Psi(x)- \T(\Psi)(x_0) \Big|\notag
\end{align}
and the two integrals in the last line are infinitesimal as $r,\rho$ tend to zero for $\mathcal H^{n-1}$-a.e. $x_0\in \partial \Omega$, respectively thanks to (\ref{Leb_cyl1}) and (\ref{Leb_cyl2}).

Let us go back to the proof of (\ref{Leb_cyl1}). Without loss of generality, we may assume that $n(x_0)=(0,0,\ldots,1)$ and that, in a neighborhood of $x_0$, the boundary $\partial \Omega$ is the graph of a $C^1$ function $h:A\to \mathbb R$, for some open set $A\subset \mathbb R^{n-1}$. More precisely, denoting by $x'$ the first $n-1$ variables of a point $x\in \mathbb R^{n}$, we can write
$$
B_\rho(x_0)\cap \partial \Omega =\{(x',h(x'))\ :\ x'\in A_\rho(x_0)\}
$$
for some open set $A_\rho(x_0)\subset \mathbb R^{n-1}$, and
$$
C^-_{r,\rho}(x_0)=\{(x',h(x')-t)\ :\ x'\in A_\rho(x_0)\,,\ t\in(0,r)\}\ .
$$
Recalling that, by definition of the extension $\tilde n$, there holds
$\tilde n(x',h(x')-t)=n(x',h(x'))$ for  $x'\in A_\rho(x_0)$, we have
\begin{align*}
 \mint_{C^-_{r,\rho}(x_0)} |\tilde{n}(x)-n(x_0)| &= \mint_{A_\rho(x_0)\times(0,r)} |\tilde{n}(x',h(x')-t) - n(x_0)|\,dx'\,dt
\\
& \leq \mint_{A_\rho(x_0)} |n(x',h(x'))- n(x_0)|\sqrt{1+|Dh|^2(x')}\,dx' \\ & = \frac{\mathcal H^{n-1}(B_\rho(x_0)\cap \partial \Omega)}{\mathcal L^{n-1}(A_\rho(x_0))} \mint_{B_\rho(x_0)\cap \partial \Omega} |n(y)- n(x_0)|\,d\mathcal H^{n-1}(y)
\\ & \leq C\,\mint_{B_\rho(x_0)\cap \partial \Omega} |n(y)- n(x_0)|\,d\mathcal H^{n-1}(y)\, ,
\end{align*}
where in the third line we used the area formula. Passing to the limit as $\rho\to 0$, we are led to (\ref{Leb_cyl1})  since by assumption $x_0$ is a Lebesgue point for $n$.

\end{proof}

\bigskip
{\bf Proof of Theorem \ref{cor_firstder}}.
Let $\u\in \mathrm{Lip}(\Omega)$ be the unique solution to problem $J(\Omega)$. Since $f$ is Gateaux differentiable except at most at the origin, the tensor $A(\u)$ is uniquely determined as in (\ref{A(ubar)}). By applying Theorem \ref{thm_firstder} and recalling that $\mathcal S$ is a singleton, we infer that
$$
J'(\Omega,V)=\int_\Omega A(\u):DV\ .
$$
Using the growth conditions (\ref{pq}) satisfied by $f$ and $g$, we see that
$A(\u)$ is in $L^\infty(\Omega;\mathbb R^{n\times n})$. Taking into account (\ref{divfree2}) in Corollary \ref{optiA}, we infer that $A(\u)$ is divergence free in the sense of distributions, in particular $A(\u)$ belongs to $\DM (\Omega;\mathbb R^{n\times n})$.
As such, since we assumed $\partial \Omega$ Lipschitz,  it admits a normal trace
$[A(\u)\,n]_{\partial \Omega} \in L^\infty(\partial \Omega;\mathbb R^n)$; moreover, by applying the generalized divergence theorem recalled in (\ref{gen_div}), we obtain (\ref{J'bis}).

\medskip

Let us now compute the normal trace $[A(\u)\,n]_{\partial \Omega}$, assuming that $\partial \Omega$ is piecewise $C^1$, that $\nabla  \ov u \in BV (\Omega)$, and that there exists $\overline\sigma \in \mathcal S^*\cap BV(\Omega;\re^n)$. Let us define
\begin{align*}
a_D(\u)& :=\nabla \u\cdot \overline \sigma - f(\nabla \u)\ ,
\\
a_N(\u)& :=-f(\nabla \u) - g(\u)\ .
\end{align*}
We remark that, by the Fenchel equality, $a_D(\u)=f^*(\overline{\sigma})$ in $\Omega$.

In the sequel, the notation $a(\u)$ is adopted for brevity in all the assertions which apply indistinctly for $a_D(\u)$ and $a_N(\u)$.

From the assumption $\u\in {\rm Lip}(\Omega)$ and
the growth conditions (\ref{pq}), we see that $a(\u)\in L^\infty(\Omega)$. We claim that $a(\u)\in BV(\Omega)$. Indeed, under the standing assumptions, $f$ and $g$ are locally Lipschitz, and the composition of a locally Lipschitz with a bounded $BV$ function is still $BV$, so that $f(\nabla \u)$ and $g(\u)$ are in $BV$. Moreover, the product of two functions which are in  $L ^ \infty \cap BV$ remains in  $L ^ \infty \cap BV$, so that the scalar product $\nabla \u\cdot \overline \sigma $ is in $BV$.
Then the claim is proved. In particular, the tensor $a(\u)I$ is an element of $\DM (\Omega;\mathbb R^{n\times n})$, and consequently
its normal trace $[a(\u)I\,n]_{\partial \Omega}$ is well defined. Moreover, according to equality (\ref{duetracce}) in Lemma \ref{prop_trace},
it can be identified with the trace of $a (\ov u)$ as a $BV$ function, namely
\begin{equation}\label{tracce=}
 \T ( a (\u)) n = [a(\u)I \,n]_{\partial \Omega}  \qquad \mathcal H ^ {n-1}\hbox{-a.e. on } \partial \Omega\, .
\end{equation}

In view of (\ref{J'bis}) and (\ref{tracce=}), in order to obtain (\ref{thesis_traces}) it is enough to show that
\begin{align*}
[A (\u)\, n - a_D(\u)I \, n ]_{\partial \Omega}  &= 0 \quad \mathrm{\ in\ case\ }(D)\ ,
\\
[A (\u)\, n - a_N(\u)I \, n]_{\partial \Omega} & = 0 \quad \mathrm{\ in\ case\ }(N)\ ,
\end{align*}
namely
\begin{align}
& \big [\big (\nabla \u  \otimes \overline \sigma- g (\u) I - \nabla \ov u \cdot \overline \sigma I  \big ) \,n \big ]_{\partial \Omega}  = 0 \quad \mathrm{\ in\ case\ }(D)\ ,\label{differenzaD}
\\
& [(\nabla \u  \otimes \overline \sigma)\, n]_{\partial \Omega}  = 0 \quad \mathrm{\ in\ case\ }(N)\ .\label{differenzaN}
\end{align}

Let us first treat the Dirichlet case. Since by assumption $\partial \Omega$ is piecewise $C ^1$, we can exploit the pointwise characterization (\ref{pointwisenormalt}) of the normal trace and  rewrite (\ref{differenzaD}) as
\begin{equation}\label{diff}
\lim_{r,\rho\to 0^+}\mint_{C^-_{r,\rho}(x_0)} \big [( \overline \sigma \cdot \tilde n) \nabla \ov u - g(\u) \tilde n
-  (\nabla \ov u  \cdot\overline \sigma) \tilde n \big ]
 = 0 \qquad \hbox{ for } \mathcal H ^ {n-1}\hbox{-a.e.\ } x_0 \in \partial \Omega\,.
\end{equation}

Recalling that $g(\u)$ is a continuous function which vanishes on $\partial \Omega$,  we have
$$\lim_{r,\rho\to 0^+}\mint_{C^-_{r,\rho}(x_0)}  [g(\u)\, \tilde n]
 = 0 \qquad \hbox{ for } \mathcal H ^ {n-1}\hbox{-a.e.\ } x_0 \in \partial \Omega\,.  $$
On the other hand, setting $P_{\tilde n}(\nabla \u):=\nabla \u - (\nabla \u\cdot \tilde n ) \tilde n$, we have
\begin{eqnarray*}
 \displaystyle{  \left| \mint_{C^-_{r,\rho}(x_0)} \big [   ( \overline \sigma \cdot \tilde n) \nabla \ov u - (\nabla \ov u  \cdot\overline \sigma) \tilde n \big ] \right| }
 &=  & \displaystyle{\left| \mint_{C^-_{r,\rho}(x_0)} \big [   ( \overline \sigma \cdot \tilde n) P_{\tilde n}(\nabla \u) - (P_{\tilde n}(\nabla \u) \cdot\overline \sigma )\tilde n \big ]  \right|}
\\ & \leq   & \displaystyle {2 \|\overline \sigma \|_{L^\infty} \mint_{C^-_{r,\rho}(x_0)}|P_{\tilde n}(\nabla \u)|}\,,
 \end{eqnarray*}
(to justify the latter inequality recall that $\overline \sigma$ belongs to $L ^ \infty(\Omega; \re ^n)$ due to the assumption $\nabla\u \in L ^ \infty (\Omega; \re ^n)$ and Lemma \ref{lemma_opt_cond} (ii)).
Therefore, in view of (\ref{diff}), claim (\ref{differenzaD}) is proved once we can show that
\begin{equation}\label{claimP}
\lim_{r,\rho\to 0^+}\mint_{C^-_{r,\rho}(x_0)} |P_{\tilde n}(\nabla \u)|= 0\ .
\end{equation}
Now we observe that, since by assumption $\u = 0$ on $\partial \Omega$ and $\nabla \u \in BV(\Omega; \re ^n)$,
the trace  $\T(\nabla \u)$ is normal to $\partial \Omega$, that is
\begin{equation}\label{traccenormali}
\T(\nabla \u)=\left(\T(\nabla \u)\cdot n \right) n\ \ \ \mathcal{H}^{n-1}-\mathrm{a.e.\ on\ }\partial \Omega\,.
\end{equation}
Indeed, thanks to the assumption that $\partial \Omega$ is piecewise $C ^ 1$, the equality (\ref{traccenormali}) can be proved by an approximation argument, which can be found for instance in  \cite[Proposition 1.4 and Section 2]{Dem} (see also \cite[Theorem 2.3]{DT}, where the same result is proved in a more general framework, allowing in particular piecewise $C^1$ boundaries). Eventually,
by (\ref{traccenormali}), we have for $\mathcal{H}^{n-1}\mathrm{-a.e.\ } x_0\in \partial \Omega$
\begin{align*}
& \mint_{C^-_{r,\rho}(x_0)}|P_{\tilde n}(\nabla \u)| \leq \mint_{C^-_{r,\rho}(x_0)}\!\!\!\! \big | \nabla \u (x)- \T(\nabla \u) (x_0)\big |   + \big |  \left(\T(\nabla \u)(x_0)\cdot n(x_0) \right) n(x_0) - (\nabla \u(x) \cdot \tilde n(x)) \tilde n(x) \big |
\\
& \quad \quad \quad \leq \mint_{C^-_{r,\rho}(x_0)}\!\!\!\! \big | \nabla \u (x)- \T (\nabla \u) (x_0)\big |   +
 \big | \T(\nabla \u)(x_0)\cdot n(x_0) - \nabla \u(x) \cdot \tilde n(x)\big |+ \|\nabla \u\|_{L^\infty}\big |  n(x_0) - \tilde n(x)\big |\,.
\end{align*}

By Lemma \ref{prop_trace} (precisely, using (\ref{Leb_cyl2}), (\ref{trace_tesi}) and (\ref{Leb_cyl1})), we infer (\ref{claimP})
and the proof of (\ref{differenzaD}) is achieved.

\medskip

Let us now consider the Neumann case. By assumption $\displaystyle
\T(\overline \sigma)\cdot n =0\quad \mathcal H ^ {n-1}\hbox{-a.e.\ on}\ \partial \Omega\,,
$ therefore by applying (\ref{trace_tesi}) we obtain
$$
\lim_{r,\rho\to 0^+}\mint_{C^-_{r,\rho}(x_0)} | \overline \sigma \cdot \tilde n|
=\lim_{r,\rho\to 0^+}\mint_{C^-_{r,\rho}(x_0)} |\overline \sigma \cdot \tilde n - \T(\overline \sigma)(x_0)\cdot n(x_0)|=0\ .
$$
Since  $\nabla \u$ is bounded, we deduce
$$
\lim_{r,\rho\to 0^+}\mint_{C^-_{r,\rho}(x_0)} ( \overline \sigma \cdot \tilde n) \nabla \ov u
 = 0 \qquad \hbox{ for } \mathcal H ^ {n-1}\hbox{-a.e.\ } x_0 \in \partial \Omega\, ,
 $$
which by the pointwise characterization (\ref{pointwisenormalt}) of the normal trace is equivalent to (\ref{differenzaN}).


\qed

\section{Appendix}\label{secapp}
\begin{proposition}\label{prop_duality}
Let $Y, Z$ be Banach spaces. Let $A:Y \to Z$ be a linear operator
with dense domain $D(A)$. Let $\Phi: Y \to \re \cup \{+ \infty\}$ be
convex, and $\Psi: Z \to \re \cup \{ + \infty\}$ be convex lower
semicontinuous. Assume there exists $u _0 \in D (A)$ such that $\Phi(u_0) < + \infty$ and $\Psi$ is continuous at $A\,u_0$.  Let $Z^*$
denote the dual space of $Z$, $A^*$ the adjoint operator of $A$, and
$\Phi ^*$, $\Psi ^*$ the Fenchel conjugates of $\Phi$, $\Psi$. Then
\begin{equation}\label{dual}
- \inf _{u \in Y} \Big \{ \Psi (A\,u) + \Phi (u)  \Big \} = \inf
_{\sigma \in Z ^*} \Big \{ \Psi ^* (\sigma) + \Phi ^* (- A ^*\,
\sigma) \Big \}\ ,
\end{equation}
and the infimum at the right hand side is achieved.

Furthermore,  $\ov{u}$ and $\ov{\sigma}$ are optimal for the l.h.s.\ and the r.h.s.\ of $(\ref{dual})$ respectively, if and
only if  there holds $\ov{\sigma}\in \partial
\Psi(A\ov{u})$ and $-A^\ast \ov{\sigma} \in \partial \Phi (\ov{u})$.
\end{proposition}

\proof See \cite[Proposition 14]{Bo}. \qed

\bigskip

\begin{proposition}\label{prop_subdif}
Let $Y$ and $Z$ be two Banach spaces and let $h:Y\times Z \to \mathbb R\cup\{+\infty\}$ be a proper function of the form
$$
h(y,z)=h_1(y)+h_2(z) \qquad \forall \, ( y , z) \in Y \times Z   \, .
$$
Then the Fenchel conjugate and the subdifferential of $h$ are given respectively by
$$
\begin{array}{ll}
& h^*(y^*,z^*)=h_1^*(y^*) + h_2^*(z^*)  \qquad \forall \, ( y ^*, z ^*) \in Y ^* \times Z ^* \, ,
 \\
&\partial h(y,z)=\partial h_1(y)\times \partial h_2(z) \qquad \forall \, ( y , z) \in Y \times Z   \, .
\end{array}
$$
\end{proposition}
\proof The statement can be easily checked by using directly the definitions of Fenchel conjugate and subdifferential, and exploiting the special structure of the function $h$. \qed

\bigskip

\begin{proposition}\label{Ih*2}
Let $h:\Omega\times \mathbb R^d\to \mathbb R$ be such that for every $x\in \Omega$ the function $h(x,\cdot)$ is lower semicontinuous and convex, and assume that there exist $\overline{v},\overline{v}^*\in L^\infty(\Omega;\re^d)$ such that
$$
\int_\Omega |h(x,\overline{v}(x))|<+\infty\ ,\ \ \int_\Omega |h^*(x,\overline{v}^*(x))|<+\infty\ ,
$$
where $h^*$ denotes the Fenchel conjugate of $h$ performed with respect to the second variable.

Let $1\leq \alpha\leq +\infty$ and consider the integral functional  $I _h$ defined on $L^\alpha(\Omega;\mathbb R^d)$ by $I _h (v):= \int _\Omega h (x,v(x))$. Then the Fenchel conjugate and the subdifferential of $I _h$ are given respectively by
$$
(I_h)^*(v^*)=\int_\Omega h^*(x,v^*(x)) \qquad \forall\, v^*\in L^{\alpha'}(\Omega;\mathbb R^d)\, ,
$$
and
$$
\partial I_h(v)=\left\{v^*\in L^{\alpha'}(\Omega;\mathbb R^d)\ :\ v^*(x)\in \partial h(x,v(x))\ a.e. \ in\ \Omega  \right\}
\qquad \forall \, v\in L^\alpha(\Omega;\mathbb R^d)
\, .
$$
\end{proposition}
%
%
\begin{proof}
See \cite[Theorem 2 and Corollary 3 of Section 3 in Chapter II]{Ek}.
\end{proof}

\bigskip

\begin{proposition}\label{clarke}
Let $Y$ be a normed space, let $h:Y\to \re$ be a convex function, and let $\overline{y}\in Y$ be a continuity point of $h$. Then $\partial h(\overline{y})$ is a nonempty and weakly * compact subset of $Y ^*$.
\end{proposition}
\begin{proof} See \cite[Proposition 2.1.2]{Cl2}. \end{proof}

\bigskip

\begin{proposition}\label{infsup}
Let $\mathcal A$ and $\mathcal B$ be nonempty convex subsets of two locally convex topological vector spaces, and let $\mathcal B$ be compact. Assume that $L:\mathcal A\times \mathcal B\to \mathbb R$ is such that for every $b\in \mathcal B$, $L(\cdot, b)$ is convex, and for every $a \in \mathcal A$, $L(a, \cdot)$ is upper semicontinuous and concave. Then, if the quantity
$$
\gamma:=\inf_{a\in \mathcal A}\,\sup_{b\in \mathcal B}\, L(a,b)
$$
is finite, we have $\gamma= \sup_{b\in \mathcal B}\,\inf_{a\in \mathcal A}\,L(a,b)$, and there exists $b^\star \in \mathcal B$ such that $\inf_{a\in \mathcal A}\,L(a,b^\star)=\gamma$.
If in addition $\mathcal A$ is compact and, for every $b\in \mathcal B$, $L(\cdot, b)$ is lower semicontinuous, there exists $a^\star \in \mathcal A$ such that $L(a^\star,b^\star)=\gamma$.
\end{proposition}

\begin{proof}
See \cite[p.~263]{Cl} and \cite{Ek2}.
\end{proof}

\end{document}